\documentclass[11pt,a4paper]{article}
\usepackage{indentfirst}
\usepackage{amsfonts}
\usepackage{amssymb}
\usepackage{mathrsfs}
\usepackage{amsmath}
\usepackage{amsthm}
\usepackage{enumerate}
\usepackage{cite}
\usepackage{mathrsfs}
\allowdisplaybreaks
\usepackage{geometry}
\usepackage{ifpdf}
\ifpdf
  \usepackage[colorlinks=true,linkcolor=blue,citecolor=red, final,backref=page,hyperindex]{hyperref}
\else
  \usepackage[colorlinks,final,backref=page,hyperindex,hypertex]{hyperref}
\fi

\usepackage{txfonts}
\usepackage{indentfirst}
\usepackage{latexsym}

\def\<{\langle}
\def\>{\rangle}

\newtheorem{thm}{Theorem}[section]
\newtheorem{lem}[thm]{Lemma}
\newtheorem{cor}[thm]{Corollary}
\newtheorem{pro}[thm]{Proposition}
\newtheorem{ex}[thm]{Example}

\newtheorem{pdef}[thm]{Proposition-Definition}

\theoremstyle{definition}
\newtheorem{defi}{Definition}[section]

\theoremstyle{remark}
\newtheorem{rmk}{Remark}[section]
\begin{document}
\title{\bf On $n$-Hom-pre-Lie superalgebras structures and their representations }
\author{\bf Othmen Ncib, Sihem Sendi}
\author{{ Othmen Ncib$^{1}$
 \footnote {  E-mail: othmenncib@yahoo.fr}
,\  Sihem Sendi$^{1}$
    \footnote {E-mail:  sihemsendi995@gmail.com}
}\\
{\small 1.  University of Gafsa, Faculty of Sciences Gafsa, 2112 Gafsa, Tunisia}}
\date{}
\maketitle

\begin{abstract}
 
 In this paper, we introduce the notion of $n$-Hom-pre-Lie superalgebras. We investigate the representation theory of $n$-Hom-pre-Lie superalgebras and we give some related results and structures based on Rota-Baxter operators, $\mathcal{O}$-operators and Nijenhuis operators. Moreover, we study relationships between $n$-Hom-pre-Lie
superalgebras and its induced Hom-pre-Lie superalgebras. The same procedure is applied for
the representations of $n$- Hom-pre-Lie superalgebras.
 
\end{abstract}

\textbf{Key words}:Hom-pre-Lie superalgebras, $n$-Hom-pre-Lie superalgebras, representation, $\mathcal O$-operators.

\textbf{Mathematics Subject Classification}: 17A42, 17B10. 

\tableofcontents

\numberwithin{equation}{section}
\section{Introduction}

The notion of $n$-Lie algebras was introduced in $1985$ by Filippov \cite{V-T-Filippov}, so it was given a classification of the $(n+1)$-dimensional $n$-Lie algebras  over an algebraically closed field of characteristic zero. The
structure of $n$-Lie algebras is very different from that of Lie algebras due to the $n$-ary
multilinear operations involved. The $n = 3$ case, i.e. $3$-ary multilinear operation, first
appeared in Nambu's work \cite{Y-Nambu} in the description of simultaneous classical dynamics of three particles. Hom-(super) generalization of this structures called $n$-Hom-Lie (super)algebras was studied in \cite{Ataguema-Makhlouf-Silvestrov,Yau-D,Ammar-Mabrouk-Makhlouf,Abdaoui-Mabrouk-Makhlouf1 } (see also \cite{Mabrouk-Ncib}, for more details). \\

Pre-Lie algebras (called also left-symmetric algebras, Vinberg algebras, quasi associative
algebras) are a class of a natural algebraic systems appearing in many
fields in mathematics and mathematical physics. They were first mentioned by Cayley
in 1890 \cite{Cayley-A} as a kind of rooted tree algebra and later arose again from the study
of convex homogeneous cones \cite{Vinberg-E-B}, affine manifold and affine structures on Lie
groups \cite{Koszul-J-L}, and deformation of associative algebras \cite{Gerstenhaber-M}. They play an important
role in the study of symplectic and complex structures on Lie groups and Lie algebras
\cite{Andrada,Chu,Dardi1,Dardi2,Lichnerowicz}, phases spaces of Lie algebras \cite{Bai1,Kupershmidt2}, certain integrable systems \cite{Bordemann},
classical and quantum Yang–Baxter equations \cite{Diata-Medina}, combinatorics \cite{Ebrahimi1}, quantum field
theory \cite{Connes-Kreimer} and operads \cite{Chapoton-Livernet}. See \cite{Bakalov-Kac},
and the survey \cite{Burde} and the references therein for more details. Recently, pre-Lie superalgebras, the $\mathbb{Z}_2$-graded version of pre-Lie algebras also appeared in many others fields; see, for example, \cite{Abdaoui-Mabrouk-Makhlouf0,Chapoton-Livernet,Gerstenhaber-M,Mikhalev}. Classifications of complexes pre-Lie superalgebras in dimensions two and three have been given recently, by Zhang and Bai \cite{Bai-Zhang}. See
\cite{Aguiar-Loday,Chen-Li,Kong-Chen-Bai,Kong-Bai} about further results. The notion of Hom-pre-Lie algebras is a twisted analog of pre-Lie algebras, where the pre-Lie algebra identity is twisted by a self linear map, called the structure map. This notion was introduced in \cite{Makhlouf-Silvestrov}. There is a close relationship between Hom-pre-
Lie algebras and Hom-Lie algebras: a Hom-pre-Lie algebra $(\mathcal{A},\circ,\alpha)$ gives rise
to a Hom-Lie algebra $(\mathcal{A},[\cdot,\cdot]^C,\alpha)$ via the commutator bracket, which is called the
subadjacent Hom-Lie algebra and denoted by $\mathcal{A}^C$. Hom-pre-Lie algebras play several roles, among them and the most important are the problems related to the representations of the Hom-Lie algebras. We can explain this in terms that the map $L:\mathcal{A}\to \mathfrak{gl}(\mathcal{A})$, defined by $L_x(y) = x\circ y$ for all $x, y\in\mathcal{A}$, gives rise to a representation of the subadjacent
Hom-Lie algebra $\mathcal{A}^C$ on $\mathcal{A}$ with respect to $\alpha\in\mathfrak{gl}(\mathcal{A})$. On the other hand, Hom-pre-Lie algebras play an important
role in the construction of Hom-Lie $2$-algebras \cite{Y-Sheng-Chen}. Recently, Hom-pre-Lie algebras were studied from several aspects: The geometrization of Hom-pre-Lie algebras was
studied in \cite{Zhang-Yu-Wang}, universal $\alpha$-central extensions of Hom-pre-Lie algebras were studied
in \cite{Sun-Chen-Zhou} and the bialgebra theory of Hom-pre-Lie algebras was studied in \cite{Sun-H-Li}. Some generalizations of pre-lie algebra have been studied, among which are given in \cite{Ming-Liu-Y-Ma}, as the authors introduce the notion of
$n$-pre-Lie algebra, which gives a $n$-Lie algebra naturally and its left multiplication operator gives rise to
a representation of this $n$-Lie algebra. An $n$-pre-Lie algebra can also be obtained through the action of a relative Rota-Baxter operator on an $n$-Lie algebra. For $(n = 3)$, see \cite{Bai-G-Sheng,Chtioui} for more details. In \cite{Shuangjian-Zhang-Shengxiang}, the authors describe the symplectic structures and phase spaces of $3$-Hom–Lie algebras from the structures of $3$-Hom–pre-Lie algebras. The same work has been introduced in the general case of n-pre-Lie algebras\cite{Hajjaji-Chtioui-Mabrouk-Makhlouf}.\\

Representations theory of different algebraic structures is an important subject of study in algebra and diverse areas. They appear in many fields of mathematics and physics. In particular, they appear in deformation and cohomology theory among other areas. In this paper, we have to talk about the representation of algebraic structures of Lie and Hom-Lie type which are introduced in several works. The notion of representation introduced for Hom–Lie algebras in\cite{Y-Sheng1}; see also\cite{Benayadi-Makhlouf}. The extension of this work to the $n$-ary and $n$-ary-Hom-super cases has been given in \cite{Ammar-Mabrouk-Makhlouf,Abdaoui-Mabrouk-Makhlouf1}. Some other extensions are given in several works such as the representation of Hom-Lie superalgebras, BiHom-Lie superalgebras, (Hom)-Lie Rinehart (super)algebras, ($n$)-(Hom)-Poisson (super)algebras... (see \cite{Luca-Vitagliano,Ncib,Hassine-Chtioui-Mabrouk-Silvestrov,Juan-Chen-Yongsheng-Cheng}, for more details). In this paper we base on the representation theory of Hom-pre-Lie algebras. This notion was introduced in \cite{Sun-H-Li} in the study of Hom-pre-Lie bialgebras (Hom-left-symmetric bialgebras). In \cite{Liu-Song-Tang}, the authors gave the natural formula of a dual representation, which is nontrivial and showed that there is well defined
tensor product of two representations of a Hom-pre-Lie algebra. A generalization of the notion of representation of pre-Lie algebras was introduced in \cite{Hajjaji-Chtioui-Mabrouk-Makhlouf}, in which the authors defined the representation of $n$-pre-Lie algebras and gave some other associated results. In this paper, we generalize this notion in the Hom-super case and we give some related results.\\

This paper is organized as follows. In Section $2$, we recall some definitions and known
results about $n$-Lie superalgebras and $n$-Hom-Lie superalgebras. We
also recall some examples for these structures. In Section $3$, we introduce the notion of $n$-Hom-pre-Lie superalgebras and their representations and we give some important results and related structures. In Section $4$, we provide a construction procedure of $n$-Hom-pre-Lie superalgebras starting from a Hom-pre-Lie superalgebra and an even $(n-2)$-linear form satisfies specific conditions. Moreover, we applied same procedure for the representations of $n$-Hom-pre-Lie superalgebras.\\

Throughout this paper, we will for simplicity of exposition assume that $\mathbb{K}$ is an algebraically closed
field of characteristic zero, even though for most of the general definitions and results in the paper this
assumption is not essential.

 A vector space $V$ is said to be a $\mathbb{Z}_2$-graded if we are given a family $(V_i)_{i\in\mathbb{Z}_2}$ of vector subspace of $V$ such that $V=V_0\oplus V_1.$  The symbol $|x|$ always implies that $x$ is a
$\mathbb{Z}_2$-homogeneous element and $|x|$ is the $\mathbb{Z}_2$-degree. In the sequel, we will denote by $\mathcal{H(A)}$ the set of all homogeneous elements of $\mathcal{A}$ and $\mathcal{H}( \mathcal{A}^n)$ refers to the set of tuples with homogeneous elements.\\

\textbf{Notations:} For any $X=(x_1,\cdots,x_n)\in\mathcal{H}( \mathcal{A}^n)$, we need the following notations
$$|X|=\displaystyle\sum_{k=1}^n|x_k|,\;\;|X|_i=\displaystyle\sum_{k=i}^n|x_k|,\;\;|X|^i=\displaystyle\sum_{k=1}^i|x_k|\;\;\text{and}\;\;|X|^j_i=\displaystyle\sum_{k=i}^j|x_k|.$$
 
\section{Preliminaries and Basics}
\label{sec:bas}

In this section, we give some preliminaries and basic results on $n$-(Hom)-Lie superalgebras and $n$-(Hom)-pre-Lie superalgebras.
\begin{defi}
  An $n$-Lie superalgebra is a pair $(\mathcal{N},[\cdot,\cdots,\cdot])$ consisting of a $\mathbb{Z}_2$-graded vector space
$\mathcal{N} = \mathcal{N}_{\overline{0}}\oplus\mathcal{N}_{\overline{1}}$ and a multilinear map $[\cdot,\cdots,\cdot] : \underbrace{\mathcal{N} \times \mathcal{N} \times\cdots \times\mathcal{N}}_{n\; times} \to \mathcal{N}$, satisfying
\begin{align}
    |[x_1,\cdots,x_n]|&=|X_n|,\label{crochet-parity}\\
    |x_1,\cdots,x_i,x_{i+1},\cdots,x_n|&=-(-1)^{|x_i||x_{i+1}|}|x_1,\cdots,x_{i+1},x_i,\cdots,x_n|,\label{SuperSkewSym}\\
    \big[x_1,\dots,x_{n-1},[y_1,\dots,y_n]\big]&= \sum_{i=1}^{n}(-1)^{|X|^{n-1}|Y|^{i-1}}\big[y_1,\dots,y_{i-1},[x_1,\dots,x_{n-1},y_i],y_{i+1},\cdots,y_n\big],\label{Nambu-identity}
\end{align}
for any $x_i,y_j\in\mathcal{H}(\mathcal{N}),\;1\leq i,j\leq n$.
\end{defi}

\begin{rmk}.
\begin{enumerate}
    \item When $\mathcal{N}_{\overline{1}}=\{0\}$,  then $\mathcal{N}$ is actually an $n$-Lie algebra.
    \item The condition \eqref{SuperSkewSym} is equivalent to
 \begin{align}\label{SuperSkewSym1}
&\left[x_1,\dots,x_i,\dots,x_j,\dots,x_n \right]=-(-1)^{|X|^{j-1}_{i+1}(|x_i|+|x_j|)+|x_i||x_j|}  \left[x_1,\dots,x_j,\dots,x_i,\dots,x_n\right],\;\forall\;1\leq i<j\leq n.
 \end{align}
\end{enumerate}
\end{rmk}

\begin{ex}
Let $\mathcal{N}=\mathcal{N}_{\overline{0}}\oplus\mathcal{N}_{\overline{1}}$ be an $(n+1)$-dimensional $\mathbb{Z}_2$-vector space, where $\mathcal{N}_{\overline{0}}=<e_1,\cdots,e_n>$ and $\mathcal{N}_{\overline{1}}=<e_{n+1}>$. Define the even super-skew-symmetric $n$-linear map $[\cdot,\cdots,\cdot]:\wedge^n\mathcal{N}\to\mathcal{N}$  by
$$[e_1,\cdots,\hat{e_i},\cdots,e_n,e_{n+1}]=e_{n+1},\;\forall 1\leq i\leq n,$$
where $\hat{e_i}$ means that the element $e_i$ is omitted. 
Then $(\mathcal{N},[\cdot,\cdots,\cdot])$ is an $n$-Lie superalgebra.

\end{ex}

\begin{defi}
 An $n$-Hom-Lie superalgebra is a tuple $(\mathcal{N},[\cdot,\cdots,\cdot],\alpha_1,\cdots,\alpha_{n-1})$ consisting of a $\mathbb{Z}_2$-graded vector space
$\mathcal{N} = \mathcal{N}_{\overline{0}}\oplus\mathcal{N}_{\overline{1}}$, a multilinear map $[\cdot,\cdots,\cdot] : \underbrace{\mathcal{N} \times \mathcal{N} \times\cdots \times\mathcal{N}}_{n\; times} \to \mathcal{N}$ and a family $(\alpha_i)_{1\leq i\leq n-1}$ of even linear maps $\alpha_i:\mathcal{N}\to\mathcal{N}$  such that the conditions \eqref{crochet-parity} and \eqref{SuperSkewSym} are satisfied and for all $X=(x_1,\cdots,x_{n-1}),Y=(y_1,\cdots,y_n)\in\mathcal{H}(\mathcal{N})$, we have
\small{\begin{equation}\label{Hom-Nambu-identity}
    \big[\alpha_1(x_1),\dots,\alpha_{n-1}(x_{n-1}),[y_1,\dots,y_n]\big]= \sum_{i=1}^{n}(-1)^{|X|^{n-1}|Y|^{i-1}}\big[\alpha_1(y_1),\dots,\alpha_{i-1}(y_{i-1}),[x_1,\dots,x_{n-1},y_i],\alpha_{i+1}(y_{i+1}),\cdots,\alpha_{n-1}(y_n)\big].
\end{equation}}

\end{defi}

We also see that if $\alpha_1=\cdots=\alpha_{n-1} = id_{\mathcal{N}}$, then $\mathcal{N}$ is just an $n$-Lie superalgebra.
\begin{defi}\rm
An \emph{n-Hom Lie superalgebra} $(N,[\cdot,\cdots,\cdot],\alpha_{1},\cdots,\alpha_{n-1})$ is \emph{multiplicative}, if
$(\alpha_{i})_{1\leq i\leq n-1}$ with $\alpha_{1}=\cdots=\alpha_{n-1}=\alpha$ and satisfying

$$\alpha([x_{1},\cdots,x_{n}])=[\alpha(x_{1}),\cdots,\alpha(x_{n})],$$
\begin{equation}\label{254}
[\alpha(x_{1}),\cdots,\alpha(x_{n-1}),[y_{1},\cdots,y_{n}]]
=\sum^{n}_{i=1}(-1)^{|X|^{n-1}|Y|^{i-1}}
[\alpha(y_{1}),\cdots,\alpha(y_{i-1}),[x_{1},\cdots,x_{n-1},y_{i}],
\alpha(y_{i+1}),\cdots,\alpha(y_{n})],
\end{equation}
for any $x_{i},y_{i} \in\mathcal{H}(\mathcal{N})$.
\end{defi}

Furthermore, if $\alpha$ is bijective then the $n$-Hom-Lie superalgebra $(\mathcal{N},[\cdot ,\cdots,\cdot ], \alpha)$ is called a regular $n$-Hom-Lie superalgebra.\\
For convenience, from now on, we always assume that $(\mathcal{N},[\cdot,\cdots,\cdot],\alpha)$ is a multiplicative n-Hom Lie superalgebra over $\mathbb{K}$ unless otherwise stated.

\begin{ex}
Let $\mathcal{N}=\mathcal{N}_{\overline{0}}\oplus\mathcal{N}_{\overline{1}}$ be an $(n+1)$-dimensional $\mathbb{Z}_2$-vector space, where $\mathcal{N}_{\overline{0}}=<e_1,\cdots,e_n>$ and $\mathcal{N}_{\overline{1}}=<e_{n+1}>$. Define the even super-skew-symmetric $n$-linear map $[\cdot,\cdots,\cdot]:\wedge^n\mathcal{N}\to\mathcal{N}$  by
$$[e_1,\cdots,\hat{e_i},\cdots,e_n,e_{n+1}]=(-1)^{i+1}e_{n+1},\;\forall 1\leq i\leq n,$$
where $\hat{e_i}$ means that the element $e_i$ is omitted
and the even linear map $\alpha:\mathcal{N}\to\mathcal{N}$ defines on the basis of $\mathcal{N}$ by
$$\alpha(e_i)=e_i,\;1\leq i\leq n\;\;\text{and}\;\;\alpha(e_{n+1})=0.$$
Then $(\mathcal{N},[\cdot,\cdots,\cdot],\alpha)$ is a multiplicative $n$-Hom-Lie superalgebra.
\end{ex}

\begin{defi}\label{defi:rep-n-Hom-Lie-sup}
A representation of an $n$-Hom-Lie superalgebra $(\mathcal{A},[\cdot,\cdots,\cdot],\alpha)$ is a triple $(V,\rho,\alpha_V)$ consisting of a
$\mathbb{Z}_2$-graded vector space $V$, an even skew-symmetric multilinear map $\rho:\mathcal{A}^{n-1}\to gl(V)$ and an even linear map $\alpha_V:V\to V$
such that for all $x_1,\cdots,x_{n-1},y_1,\cdots,y_n\in\mathcal{H}(\mathcal{A})$, we have
\begin{align}
&\rho(\alpha(x_1),\cdots,\alpha(x_{n-1}))\alpha_V=\alpha_V\rho(x_1,\cdots,x_{n-1}),\label{compatible-rho-alpha}\\
    &\rho(\alpha(x_1),\cdots,\alpha(x_{n-1}))\rho(y_1,\cdots,y_{n-1})-(-1)^{|X|^{n-1}|Y|^{n-1}}
    \rho(\alpha(y_1),\cdots,\alpha(y_{n-1}))\rho(x_1,\cdots,x_{n-1})\nonumber\\&=\displaystyle\sum_{i=1}^{n-1}(-1)^{|X|^{n-1}|Y|^{i-1}}
    \rho(\alpha(y_1),\cdots,\alpha(y_{i-1}),[x_1,\cdots,x_{n-1},y_i],\alpha(y_{i+1}),\cdots\alpha(y_{n-1}))\alpha_V,\label{repr-n-Hom-Lie1}\\&
    \rho(\alpha(x_1),\cdots,\rho(x_{n-2}),[y_1,\cdots,y_n])\alpha_V\nonumber\\&=
    \displaystyle\sum_{i=1}^{n}(-1)^{n-i}(-1)^{|X|^{n-2}(|Y|+|y_i|)+|y_i||Y|_{i+1}}\rho(\alpha(y_1),\cdots,\hat{y_i},\cdots,\alpha(y_n))\rho(x_1,\cdots,x_{n-2},y_i)\label{repr-n-Hom-Lie2}.
\end{align}
\end{defi}
\begin{ex}\label{alpha-s-adj-rep}
Defining for any integer $s\geq0$ the $\alpha^s$-adjoint representation of an $n$-Hom-Lie superalgebra $(\mathcal{N},[\cdot,\cdots,\cdot],\alpha)$ on $\mathcal{N}^{\otimes n-1}$ as follows
$$ad^s_{x_1,\cdots,x_{n-1}}(x)=[\alpha^s(x_1),\cdots,\alpha^s(x_{n-1}),x]\;\;\text{for all}\;x_i,x\in\mathcal{H}(\mathcal{N}),\;1\leq i\leq n-1.$$
Let us denote the $\alpha^s$-adjoint representation of the $n$-Hom-Lie superalgebra $(\mathcal{N},[\cdot,\cdots ,\cdot ],\alpha)$ by the triple $(\mathcal{N},ad^s,\alpha)$. We also denote $ad^0_{x_1,\cdots,x_{n-1}}$ simply by $ad_{x_1,\cdots,x_{n-1}}$ for any
$x_1,\cdots,x_{n-1}\in\mathcal{H}(\mathcal{N})$.
\end{ex}
\begin{ex}\label{dual-representation}
Let $(\mathcal{N},[\cdot,\cdots,\cdot],\alpha)$ be a regular $n$-Hom-Lie superalgebra and $(V,\rho,\alpha_V)$ be an $n$-Hom-Lie superalgebra representation with $\alpha_V$ being an invertible linear map. Define $\rho^*:\mathcal{N}^{n-1}\to End(V^*)$ as usual by
$$<\rho^*(x_1,\cdots,x_{n-1})(\xi),u>=-<\xi,\rho(x_1,\cdots,x_{n-1})(u)>,\;\forall x_i\in\mathcal{H}(\mathcal{N}),1\leq i\leq n-1,\;u\in\mathcal{H}(V),\;\xi\in V^*.$$
However, in general $\rho^*$ is not a representation of $\mathcal{N}$ anymore. 
Let us define the map $\rho^\star:\mathcal{N}^{n-1}\to End(V^\ast)$ by
\begin{align}\label{dual-representation}
    <\rho^\star(x_1,\cdots,x_{n-1})(\xi),v>:&=<\rho^\ast(\alpha(x_1),\cdots,\alpha(x_{n-1}))((\alpha_V^{-2})^{\ast}(\xi)),v>\\\nonumber&=-<\xi,\rho(\alpha^{-1}(x_1),\cdots,\alpha^{-1}(x_{n-1}))(\alpha_V^{-2}(v))>,
\end{align}
for all $\xi\in V^\ast,\;x_1,\cdots,x_{n-1}\in\mathcal{H}(\mathcal{N})$ and $v\in V$. Then, the triple $(V^\ast,\rho^\star,(\alpha_V^{-1})^*)$  is a representation of the $n$-Hom-Lie superalgebra $(\mathcal{N},[\cdot,\cdots,\cdot],\alpha)$ on the dual vector space $V^\ast$ with
respect to the map $(\alpha_V^{-1})^\ast$. This is also known as the “dual representation” to $(V,\rho,\alpha_V)$.\\
In particular, let us also recall that the “coadjoint representation” of a regular $n$-Hom-Lie superalgebra $(\mathcal{N},[\cdot,\cdots,\cdot],\alpha)$ on $\mathcal{N}^\ast$ with respect to $(\alpha^{-1})^\ast$ 
is given by the triple $(\mathcal{N}^\ast,ad^\star,(\alpha^{-1})^\ast)$, where 
\begin{align}\label{dual-adjoint-rep}
  <ad^\star(x_1,\cdots,x_{n-1})(\xi),x>&=-<\xi,ad(\alpha^{-1}(x_1),\cdots,\alpha^{-1}(x_{n-1}))(\alpha^{-2}(x))>\\&=-<\xi,[\alpha^{-1}(x_1),\cdots,\alpha^{-1}(x_{n-1}),\alpha^{-2}(x)]> ,
\end{align}
for all $x_i,x\in \mathcal{H}(\mathcal{N}),\;1\leq i\leq n-1,\;\xi\in\mathcal{N}^\ast$.
\end{ex}
\begin{pro}(\cite{Ncib})\label{direct-sum-n-hom-lie-superalgebras}
Let $(\mathcal{N},[\cdot,\cdots,\cdot],\alpha)$ be an $n$-Hom-Lie superalgebra, $\rho:\Lambda^{n-1}\mathcal{N}\to gl(V)$  and $\alpha_V:V\to V$ are two even linear maps. Then $(\mathcal{N}\oplus V,[\cdot,\cdots,\cdot]_{\mathcal{N}\oplus V},\alpha+\alpha_V)$ is an $n$-Hom-Lie superalgebra if and only if $(V,\rho,\alpha_V)$ is a representation of $(\mathcal{N},[\cdot,\cdots,\cdot],\alpha)$, where 
\begin{equation}\label{crochet-direct-sum-n-hom-lie-superalgebras}
    [x_1+a_1,\cdots,x_n+a_n]_{\mathcal{N}\oplus V}=[x_1,\cdots,x_n]+\displaystyle\sum_{k=1}^n(-1)^{|x_k||X|_{k+1}}\rho(x_1,\cdots,\hat{x_k},\cdots,x_n)a_k,
\end{equation}
for all $x_1,\cdots,x_n\in\mathcal{H}(\mathcal{N})$ and $a_1,\cdots,a_n\in\mathcal{H}(V)$.
\end{pro}

\begin{defi}
Let $(\mathcal{N},[\cdot,\cdots,\cdot],\alpha)$ be an $n$-Hom-Lie superalgebra and $s$ be a non-negative integer. Then, an even linear operator $\mathcal{R} :\mathcal{N}\to\mathcal{N}$ is called an
$s$-Rota–Baxter operator of weight $\lambda$ on  $(\mathcal{N},[\cdot,\cdots,\cdot],\alpha)$ if $\mathcal{R}\circ\alpha=\alpha\circ\mathcal{R}$ and the following identity is satisfied:
\begin{equation}
[ R(x_1), \dots, R(x_n)]
=R\Big( \sum\limits_{\emptyset \neq I\subseteq [n]}\lambda^{|I|-1} [ \hat{R}(x_1), \dots, \hat{R}(x_i), \dots, \hat{R}(x_{n})]\Big),
\end{equation}
where
$\hat{R}(x_i):=\hat{R}_I(x_i):=\left\{\begin{array}{ll} x_i, & i\in I, \\ \alpha^s R(x_i), & i\not\in I \end{array}\right. \text{ for all } x_1,\dots, x_n\in \mathcal{N}.
$\\

For $\alpha=Id$, then we recover the notion of Rota–Baxter operators on an $n$-Lie superalgebra.
\end{defi}

\begin{defi}\label{defi:O-operator}
Let $(\mathcal{N},[\cdot,\cdots,\cdot],\alpha)$ be an $n$-Hom-Lie superalgebra and $(V,\rho,\alpha_V)$
a representation.  An even linear map $T:V\rightarrow \mathcal{N}$ is called
an $\mathcal O$-operator associated to $( V,\rho,\alpha_V)$ if $T$
satisfies
\begin{equation}
 \alpha\circ T=T\circ\alpha_V,   
\end{equation}
\begin{equation}\label{eq:Ooperator}
 [T(u_1),\cdots,T(u_n)]=T\Big(\sum_{i=1}^{n}(-1)^{n-i}(-1)^{|u_i||U|_{i+1}}\rho(T(u_1),\cdots,\widehat{T(u_i)},\cdots,T(u_n))(u_i)\Big), 
\end{equation}
for all $u_i\in \mathcal{H}(V),\;1\leq i\leq n$.
An $\mathcal{O}$-operator associated to
the adjoint representation $(A,ad,\alpha)$ is called a Rota-Baxter operator of weight $\lambda=0$.
\end{defi}
\begin{rmk}
Recall the $\alpha^s$-adjoint representation $(\mathcal{N},ad^s,\alpha)$ of an $n$-Hom-Lie superalgebra $(\mathcal{N},[\cdot,\cdots,\cdot],\alpha)$  for any integer $s\geq0$ given in Example \ref{alpha-s-adj-rep}. Then, an $s$-Rota–Baxter operator of weight $0$ on the $n$-Hom-Lie superalgebra $(\mathcal{N},[\cdot,\cdots,\cdot],\alpha)$ is an $\mathcal{O}$-operator on $(\mathcal{N},[\cdot,\cdots,\cdot],\alpha)$ with respect to the representation $(\mathcal{N},ad^s,\alpha)$. Thus, the notion of $\mathcal{O}$-operators is a generalization of Rota–Baxter operators and therefore also known as relative
or generalized Rota–Baxter operators.

\end{rmk}
\begin{ex}
Let $(V,\rho,\alpha_V)$ be a representation of an $n$-Hom-Lie Superalgebra $(\mathcal{N},[\cdot,\cdots,\cdot],\alpha)$ and $T:V\rightarrow \mathcal{N}$ is
an $\mathcal O$-operator associated to $( V,\rho,\alpha_V)$. A pair $(\phi_{\mathcal{N}},\phi_V)$ is an endomorphism of the $\mathcal{O}$-operator $T$ if
\begin{align*}
T\circ \phi_{V}&=\phi_{\mathcal{N}}\circ T\quad \mbox{and}\\
\rho(\phi_{\mathcal{N}}(x_1),\cdots,\phi_{\mathcal{N}}(x_{n-1})))(\phi_{V}(v))&=\phi_V(\rho(x_1,\cdots,x_{n-1})(v)), \quad \mbox{for all } x_i\in \mathcal{N},\;1\leq i\leq n-1, ~v\in V.
\end{align*}

 Let us consider the $n$-Hom-Lie superalgebra $(\mathcal{N},[\cdot,\cdots,\cdot,]_{\phi_\mathcal{N}},\phi_{\mathcal{N}})$ obtained by composition, where the $n$-Hom-Lie bracket is given by
  $$[\cdot,\cdots,\cdot,]_{\phi_\mathcal{N}}:=\phi_{\mathcal{N}}\circ [\cdot,\cdots,\cdot].$$
If we consider the composition $\rho_{\phi_V}:=\phi_V\circ \rho$, then the triple $(V,\rho_{\phi_V},\phi_V)$ is an $n$-Hom-Lie superalgebra representation of $(\mathcal{N},[\cdot,\cdots,\cdot]_{\phi_\mathcal{N}},\phi_{\mathcal{N}})$. Moreover,
\begin{align*}
    [T(u_1),\cdots,T(u_n)]_{\phi_\mathcal{N}}&=\phi_{\mathcal{N}}([T(u_1),\cdots,T(u_n)])\\&=\phi_{\mathcal{N}}\Big(T\Big(\sum_{i=1}^{n}(-1)^{n-i}(-1)^{|u_i||U|_{i+1}}\rho(T(u_1),\cdots,\widehat{T(u_i)},\cdots,T(u_n))(u_i)\Big)\Big)\\&=T\Big(\sum_{i=1}^{n}(-1)^{n-i}(-1)^{|u_i||U|_{i+1}}\rho_{\phi_V}(T(u_1),\cdots,\widehat{T(u_i)},\cdots,T(u_n))(u_i)\Big)
\end{align*}

for all $u_i \in V,\;1\leq i\leq n$. Clearly, it follows that the map $T:V\rightarrow \mathcal{N}$ is an $\mathcal{O}$-operator on the $n$-Hom-Lie superalgebra $(\mathcal{N},[\cdot,\cdots,\cdot]_{\phi_\mathcal{N}},\phi_{\mathcal{N}})$ with respect to the $n$-Hom-Lie superalgebra representation $(V,\rho_{\phi_V},\phi_V)$. 
\end{ex}

In the following, we give a characterization of an $\mathcal{O}$-operator $T$ in terms of an $n$-Hom-Lie subalgebra structure on the graph of $T$ defined by
$$Gr(T)=\{(T(u),u)/\;u\in V\}.$$
\begin{pro}
An even linear map $T:\mathcal{N}\to V$ is an $\mathcal{O}$-operator on the $n$-Hom-Lie superalgebra $(\mathcal{N},[\cdot,\cdots,\cdot],\alpha)$ with respect to the representation $(V,\rho,\alpha_V)$ if and only if $Gr(T)$ is an $n$-Hom-Lie subalgebra of the semi-direct product $n$-Hom-Lie superalgebra $(\mathcal{N}\oplus V,[\cdot,\cdots,\cdot]_{\mathcal{N}\oplus V},\alpha+\alpha_V)$, defined in Proposition \ref{direct-sum-n-hom-lie-superalgebras}.
\end{pro}
\begin{proof}
Let $(Tu_k,u_k)\in Gr(T),\;1\leq k\leq n$. Then, if $T$ is an $\mathcal{O}$-operator on $(\mathcal{N},[\cdot,\cdots,\cdot],\alpha)$, we have
\begin{align*}
 [(Tu_1,u_1),\cdots,(Tu_n,u_n)]_{\mathcal{N}\oplus V} &= [Tu_1,\cdots,Tu_n],\displaystyle\sum_{k=1}^n(-1)^{|u_k||U|_{k+1}}\rho(Tu_1,\cdots,\widehat{Tu_k},\cdots,Tu_n)u_k\\&=T\big(\displaystyle\sum_{k=1}^n(-1)^{n-k}(-1)^{|u_k||U|_{k+1}}\rho(Tu_1,\cdots,\widehat{Tu_k},\cdots,Tu_n)u_k\big),\\&\displaystyle\sum_{k=1}^n(-1)^{|u_k||U|_{k+1}}\rho(Tu_1,\cdots,\widehat{Tu_k},\cdots,Tu_n)u_k\in Gr(T), 
\end{align*}
which implies that $Gr(T)$ is a subalgebra of the semi-direct product $n$-Hom-Lie superalgebra $(\mathcal{N}\oplus V,[\cdot,\cdots,\cdot]_{\mathcal{N}\oplus V},\alpha+\alpha_V)$.\\

In the other hand, if $Gr(T)$ is a subalgebra of the semi-direct product $n$-Hom-Lie superalgebra $(\mathcal{N}\oplus V,[\cdot,\cdots,\cdot]_{\mathcal{N}\oplus V},\alpha+\alpha_V)$, then we have 
\begin{align*}
 [(Tu_1,u_1),\cdots,(Tu_n,u_n)]_{\mathcal{N}\oplus V} &= [Tu_1,\cdots,Tu_n],\displaystyle\sum_{k=1}^n(-1)^{|u_k||U|_{k+1}}\rho(Tu_1,\cdots,\widehat{Tu_k},\cdots,Tu_n)u_k\in Gr(T),    
\end{align*}
which gives that $[Tu_1,\cdots,Tu_n] =T\Big( \displaystyle\sum_{k=1}^n(-1)^{|u_k||U|_{k+1}}\rho(Tu_1,\cdots,\widehat{Tu_k},\cdots,Tu_n)u_k\Big)$. Therefore $T$ is an $\mathcal{O}$-operator on $(\mathcal{N},[\cdot,\cdots,\cdot],\alpha)$.
\end{proof}

It is of course that there are some other characterizations of the $\mathcal{O}$-operators on an $n$-Hom-Lie superalgebras, among them and this most interesting the characterization in term of a Nijenhuis operators. In the following Proposition, we characterize $\mathcal{O}$-operators on $n$-Hom-Lie superalgebras in terms of the Nijenhuis operators. In the following, we need to define the Nijenhuis operator on an $n$-Hom-Lie superalgebras $(\mathcal{N},[\cdot,\cdots,\cdot],\alpha)$, as an even linear map $N:\mathcal{N}\to\mathcal{N}$ which satisfies the following identity
\begin{equation}\label{Nijenhuis-identity}
[N(x_1),\cdots,N(x_n)]=N\big( \sum\limits_{\emptyset \neq I\subseteq [n]}N^{|I|-1} [ \hat{N}(x_1), \dots, \hat{N}(x_i), \dots, \hat{N}(x_{n})]\big),    
\end{equation}
where
$\hat{N}(x_i):=\hat{N}_I(x_i):=\left\{\begin{array}{ll} x_i, & i\in I, \\ N(x_i), & i\not\in I \end{array}\right. \text{ for all } x_1,\dots, x_n\in \mathcal{H}(\mathcal{N}).$
\begin{pro}
Let $(V,\rho,\alpha_V)$ be a representation of an $n$-Hom-Lie superalgebra $(\mathcal{N},[\cdot,\cdots,\cdot],\alpha)$ and $T:V\to\mathcal{N}$ an even linear map. Then $T$ is an $\mathcal{O}$-operator on $(\mathcal{N},[\cdot,\cdots,\cdot],\alpha)$ with respect to $(V,\rho,\alpha_V)$ if and only if the operator 
$$N_T=\begin{bmatrix}
   0 & T \\
    0  & 0
\end{bmatrix}: \mathcal{N}\oplus V\rightarrow \mathcal{N}\oplus V $$
is a Nijenhuis operator on the semi-direct product $n$-Hom-Lie superalgebra $(\mathcal{N}\oplus V,[\cdot,\cdots,\cdot]_{\mathcal{N}\oplus V},\alpha+\alpha_V)$.
\end{pro}
\begin{proof}
By using the Definition of the map $N_T$ and the bracket $[\cdot,\cdots,\cdot]_{\mathcal{N}\oplus V}$, we have
\begin{align*}
[N_T(x_1+u_1),\cdots,N_T(x_n+u_n)]_{\mathcal{N}\oplus V}&=[T(u_1)+0,\cdots,T(u_n)+0]\\&=[T(u_1),\cdots,T(u_n)],
\end{align*}
and by the obvious result $N_T^k=0,\;\forall k\geq2$, we have
\begin{align*}
&N_T\big( \sum\limits_{\emptyset \neq I\subseteq [n]}N_T^{|I|-1} [ \hat{N_T}(x_1+u_1), \dots, \hat{N_T}(x_i+u_i), \dots, \hat{N_T}(x_n+u_n)]_{\mathcal{N}\oplus V}\big)\\&= N_T\big(\displaystyle\sum_{i=1}^n[N_T(x_1+u_1),\cdots,x_i+u_i,\cdots,N_T(x_n+u_n)]_{\mathcal{N}\oplus V}\big)\\&=N_T\big(\displaystyle\sum_{i=1}^n[T(u_1),\cdots,x_i+u_i,\cdots,T(u_n)]_{\mathcal{N}\oplus V}\big)\\&=N_T\big(\displaystyle\sum_{i=1}^n[T(u_1),\cdots,x_i,\cdots,T(u_n)]+\displaystyle\sum_{i=1}^n(-1)^{|u_i||U|_{i+1}}\rho(T(u_1),\cdots,\hat{T(u_i)},\cdots,T(u_n))(u_i)\big)\\&=  T\big(\displaystyle\sum_{i=1}^n(-1)^{|u_i||U|_{i+1}}\rho(T(u_1),\cdots,\hat{T(u_i)},\cdots,T(u_n))(u_i)\big), 
\end{align*}
for all $x_i\in\mathcal{H}(\mathcal{N}),\;u_i\in\mathcal{H}(V)$. By a direct computation, we conclude that
\begin{align*}
&N_T\big( \sum\limits_{\emptyset \neq I\subseteq [n]}N_T^{|I|-1} [ \hat{N_T}(x_1+u_1), \dots, \hat{N_T}(x_i+u_i), \dots, \hat{N_T}(x_n+u_n)]_{\mathcal{N}\oplus V}\big)\\&=N_T\big( \sum\limits_{\emptyset \neq I\subseteq [n]}N_T^{|I|-1} [ \hat{N_T}(x_1+u_1), \dots, \hat{N_T}(x_i+u_i), \dots, \hat{N_T}(x_n+u_n)]_{\mathcal{N}\oplus V}\big)
\end{align*}
if and only if 
$$[T(u_1),\cdots,T(u_n)]=T\big(\displaystyle\sum_{i=1}^n(-1)^{|u_i||U|_{i+1}}\rho(T(u_1),\cdots,\hat{T(u_i)},\cdots,T(u_n))(u_i)\big),$$
for all $x_i\in\mathcal{H}(\mathcal{N}),\;u_i\in\mathcal{H}(V)$, which gives the result.
\end{proof}

\section{$n$-Hom-pre-Lie superalgebras and their representations}

In \cite{Burde}, the author introduced the notion of pre-Lie algebras and given their representation, some other practical results are also studied, among those which are most interesting the cohomology and deformations of pre-Lie algebras. This notion has been extended in more general cases (for more details see \cite{Hajjaji-Chtioui-Mabrouk-Makhlouf}). In this section we introduce the notion of $n$-Hom-pre-Lie superalgebras and define their representation also we give some algebraic structures and results concerning this notion.

\subsection{$n$-Hom-pre-Lie superalgebras}

\begin{defi}
An $n$-Hom-pre-Lie superalgebra is a triple $(\mathcal{A},\{\cdot,\cdots,\cdot\},\alpha)$ consisting of a $\mathbb{Z}_2$-graded vector space $\mathcal{A}$, an even multilinear map $\{\cdot,\cdots,\cdot\}:\wedge^n\mathcal{A}\to\mathcal{A}$ super-skew-symmetric on the first $(n-1)$ terms and an even linear map $\alpha:\mathcal{A}\to\mathcal{A}$ such that for all $x_i,y_i\in\mathcal{H}(\mathcal{A}),\;1\leq i\leq n$, the following identities are satisfied:
{\small\begin{eqnarray}
\{\alpha(x_1),\cdots,\alpha(x_{n-1}),\{y_1,\cdots,y_n\}\}&=&\sum_{i=1}^{n-1}(-1)^{|Y|^{i-1}|X|^{n-1}}\{\alpha(y_1),\cdots,\alpha(y_{i-1}),[x_1,\cdots,x_{n-1},y_i]^C,\alpha(y_{i+1}),\cdots,\alpha(y_n)\}\nonumber\\
&&+(-1)^{|Y|^{n-1}|X|^{n-1}}\{\alpha(y_1),\cdots,\alpha(y_{n-1}),\{x_1,\cdots,x_{n-1},y_n\}\},\label{n-Hom-pre-Lie-sup 1}\\
\{ [x_1,\cdots,x_n]^C,\alpha(y_1),\cdots, \alpha(y_{n-1})\}&=&\sum_{i=1}^{n}(-1)^{n-i}(-1)^{|x_i||X|^n_{i+1}}\{\alpha(x_1),\cdots,\widehat{x}_i,\cdots,\alpha(x_{n}),\{ x_i,y_1,\cdots,y_{n-1}\}\},\label{n-Hom-pre-Lie-sup 2}
\end{eqnarray}}
 where  $[\cdot,\cdots,\cdot]^C$ is defined by
\begin{equation}
[x_1,\cdots,x_n]^C=\sum_{i=1}^{n}(-1)^{n-i}(-1)^{|x_i||X|^n_{i+1}}\{x_1,\cdots,\widehat{x_i},\cdots,x_n,x_i\},\quad \forall  x_i\in \mathcal{H}(\mathcal{A}),1\leq i\leq n.\label{eq:ncc}
\end{equation}
\end{defi}

\begin{pdef}\label{pro:subadj}
Let $(\mathcal{A},\{\cdot,\cdots,\cdot\},\alpha)$ be an $n$-Hom-pre-Lie superalgebra. Then  $(\mathcal{A},[\cdot,\cdots,\cdot]^C,\alpha)$, where $[\cdot,\cdots,\cdot]^C$ is given by Eq.~\eqref{eq:ncc}
is an $n$-Hom-Lie superalgebra called the sub-adjacent $n$-Hom-Lie superalgebra of $(\mathcal{A},\{\cdot,\cdots,\cdot\},\alpha)$, and denoted by $\mathcal{A}^{c}$. $(\mathcal{A},\{\cdot,\cdots,\cdot\},\alpha)$ is called a compatible
$n$-Hom-pre-Lie superalgebra of the $n$-Hom-Lie superalgebra $\mathcal{A}^{c}$.
\end{pdef}
\begin{proof}
Let $x_i,y_i\in\mathcal{H}(\mathcal{A}),\;1\leq i\leq n$. For all $1\leq k\leq n-1$, then by using the definition of $[\cdot,\cdots,\cdot]^C$, we have:
\begin{align*}
 [x_1,\cdots,x_k,x_{k+1},\cdots,x_{n-1},x_n]^C&=\sum_{i=1}^{n}(-1)^{n-i}(-1)^{|x_i||X|^n_{i+1}}\{x_1,\cdots,\widehat{x_i}\cdots,x_k,x_{k+1},\cdots,x_n,x_i\}\\&+ \sum_{i=1}^{n}(-1)^{n-i}(-1)^{|x_i||X|^n_{i+1}}\{x_1,\cdots,x_k,x_{k+1},\cdots,\widehat{x_i}\cdots,x_n,x_i\}\\&+(-1)^{n-k}(-1)^{|x_k||X|^n_{k+1}}\{x_1,\cdots,x_{k-1},x_{k+1},\cdots,x_n,x_k\}\\&+(-1)^{n-k-1}(-1)^{|x_{k+1}||X|^n_{k+2}}\{x_1,\cdots,x_k,x_{k+2},\cdots,x_n,x_{k+1}\}\\&=-(-1)^{|x_k||x_{k+1}|}\Big(\sum_{i=1}^{n}(-1)^{n-i}(-1)^{|x_i||X|^n_{i+1}}\{x_1,\cdots,\widehat{x_i}\cdots,x_{k+1},x_k,\cdots,x_n,x_i\}\\&+ \sum_{i=1}^{n}(-1)^{n-i}(-1)^{|x_i||X|^n_{i+1}}\{x_1,\cdots,x_{k+1},x_k,\cdots,\widehat{x_i}\cdots,x_n,x_i\}\\&+(-1)^{n-k-1}(-1)^{|x_k||X|^n_{k+2}}\{x_1,\cdots,x_{k-1},x_{k+1},\cdots,x_n,x_k\}\\&+(-1)^{n-k}(-1)^{|x_{k+1}|(|x_k|+|X|^n_{k+2})}\{x_1,\cdots,x_k,x_{k+2},\cdots,x_n,x_{k+1}\}\Big)\\&=-(-1)^{|x_k||x_{k+1}|}[x_1,\cdots,x_{k+1},x_k,\cdots,x_{n-1},x_n]^C,  
\end{align*}
which implies that $[\cdot,\cdots,\cdot]^C$ is super-skew-symmetric. It remains to show that $[\cdot,\cdots,\cdot]^C$ satisfies condition \eqref{254}.\\ On the one hand, we have
\begin{align*}
M&=[\alpha(x_1),\cdots,\alpha(x_{n-1}),[y_1,\cdots,y_n]^C]^C\\&=\displaystyle\sum_{i=1}^{n-1}(-1)^{n-i}(-1)^{|x_i|(|X|^{n-1}_{i+1}+|Y|)}\{\alpha(x_1),\cdots,\hat{\alpha(x_i)},\cdots,\alpha(x_{n-1}),[y_1,\cdots,y_n]^C,\alpha(x_i)\}\\&+ \{\alpha(x_1),\cdots,\alpha(x_{n-1}),[y_1,\cdots,y_n]^C\}\\&=\displaystyle\sum_{i=1}^{n-1}\displaystyle\sum_{j=1}^{n-1}(-1)^{i+j}(-1)^{|x_i|(|X|^{n-1}_{i+1}+|Y|)}(-1)^{|y_j||Y|^n_{j+1}}\{\alpha(x_1),\cdots,\hat{\alpha(x_i)},\cdots,\alpha(x_{n-1}),\{y_1,\cdots,\hat{y_j},\cdots,y_n,y_j\},\alpha(x_i)\}\\&+\displaystyle\sum_{i=1}^{n-1}(-1)^{n-i}(-1)^{|x_i|(|X|^{n-1}_{i+1}+|Y|)} \{\alpha(x_1),\cdots,\hat{\alpha(x_i)},\cdots,\alpha(x_{n-1}),\{y_1,\cdots,\cdots,y_n\},\alpha(x_i)\}\\&+ \displaystyle\sum_{j=1}^{n-1}(-1)^{n-j}(-1)^{|y_j||Y|^n_{j+1}} \{\alpha(x_1),\cdots,\alpha(x_{n-1}),\{y_1,\cdots,\hat{y_j},\cdots,y_n,y_j\}\}+\{\alpha(x_1),\cdots,\alpha(x_{n-1}),\{y_1,\cdots,\cdots,y_n\}\}.
\end{align*}
On the other hand, we have
\begin{align*}
N&= \displaystyle\sum_{i=1}^n(-1)^{|X|^{n-1}|Y|^{i-1}}[\alpha(y_1),\cdots,\alpha(y_{i-1}),[x_1,\cdots,x_{n-1},y_i]^C,\alpha(y_{i+1}),\cdots,\alpha(y_n)]^C\\&= \displaystyle\sum_{i=1}^n\displaystyle\sum_{j=1}^{i-1}(-1)^{n-j}(-1)^{|X|^{n-1}|Y|^{i-1}}(-1)^{|y_j|(|X|^{n-1}+|Y|^n_{j+1})}\{\alpha(y_1),\cdots,\hat{\alpha(y_j)},\cdots,[x_1,\cdots,x_{n-1},y_i]^C,\cdots,\alpha(y_n),\alpha(y_j)\} \\&+\displaystyle\sum_{i=1}^n\displaystyle\sum_{j=i+1}^{n-1}(-1)^{n-j}(-1)^{|X|^{n-1}|Y|^{i-1}}(-1)^{|y_j||Y|^n_{j+1}}\{\alpha(y_1),\cdots,[x_1,\cdots,x_{n-1},y_i]^C,\cdots,\hat{\alpha(y_j)},\cdots,\alpha(y_n),\alpha(y_j)\}\\&+\displaystyle\sum_{i=1}^n(-1)^{n-i}(-1)^{|X|^{n-1}|Y|^{i-1}}(-1)^{|Y|^n_{i+1}(|X|^{n-1}+|y_i|)}\{\alpha(y_1),\cdots,\alpha(y_{i-1}),\alpha(y_{i+1}),\cdots,\alpha(y_n),[x_1,\cdots,x_{n-1},y_i]^C\}\\&+\displaystyle\sum_{i=1}^n(-1)^{|X|^{n-1}|Y|^{i-1}}\{\alpha(y_1),\cdots,\alpha(y_{i-1}),[x_1,\cdots,x_{n-1},y_i]^C,\alpha(y_{i+1}),\cdots,\alpha(y_n)\}\\&=\displaystyle\sum_{i=1}^n\displaystyle\sum_{j=1}^{i-1}(-1)^{n-j}(-1)^{|X|^{n-1}|Y|^{i-1}}(-1)^{|y_j|(|X|^{n-1}+|Y|^n_{j+1})}\{\alpha(y_1),\cdots,\hat{\alpha(y_j)},\cdots,[x_1,\cdots,x_{n-1},y_i]^C,\cdots,\alpha(y_n),\alpha(y_j)\} \\&+\displaystyle\sum_{i=1}^n\displaystyle\sum_{j=i+1}^{n-1}(-1)^{n-j}(-1)^{|X|^{n-1}|Y|^{i-1}}(-1)^{|y_j||Y|^n_{j+1}}\{\alpha(y_1),\cdots,[x_1,\cdots,x_{n-1},y_i]^C,\cdots,\hat{\alpha(y_j)},\cdots,\alpha(y_n),\alpha(y_j)\}\\&+\displaystyle\sum_{i=1}^{n-1}(-1)^{n-i}(-1)^{|X|^{n-1}|Y|^{i-1}}(-1)^{|Y|^n_{i+1}(|X|^{n-1}+|y_i|)}\{\alpha(y_1),\cdots,\alpha(y_{i-1}),\alpha(y_{i+1}),\cdots,\alpha(y_n),[x_1,\cdots,x_{n-1},y_i]^C\}\\&+\displaystyle\sum_{i=1}^{n-1}(-1)^{n-i}(-1)^{|X|^{n-1}|Y|^{i-1}}(-1)^{|Y|^n_{i+1}(|X|^{n-1}+|y_i|)}\{\alpha(y_1),\cdots,\alpha(y_{n-1}),\{x_1,\cdots,\hat{x_i}\cdots,x_{n-1},y_n,x_i\}\}\\&+\{\alpha(y_1),\cdots,\alpha(y_{n-1}),\{x_1,\cdots,x_{n-1},y_n\}\}+\displaystyle\sum_{i=1}^n(-1)^{|y_i||Y|_{i-1}}\{[x_1,\cdots,x_{n-1},y_i]^C,\alpha(y_1),\cdots,\alpha(y_{i-1}),\alpha(y_{i+1}),\cdots,\alpha(y_n)\}.
\end{align*}
Using the identities \eqref{n-Hom-pre-Lie-sup 1}-\eqref{n-Hom-pre-Lie-sup 2} and by a direct computation, we find $M-N=0$, which implies that $[\cdot,\cdots,\cdot]^C$ gives an $n$-Hom-Lie superalgebra structure on $\mathcal{A}$.

\end{proof}
Let $(\mathcal{A},\{\cdot,\cdots,\cdot\},\alpha)$ be an $n$-Hom-pre-Lie superalgebra. Defining the two even multiplications $L,R: \wedge^{n-1}\mathcal{A}\rightarrow  gl(\mathcal{A})$
by
\begin{equation}\label{eq:R}
L(x_1,\cdots,x_{n-1})x_n=\{x_1,\cdots,x_{n-1},x_n\},
\end{equation}
 and
\begin{equation}
    R(x_1,\cdots,x_{n-1})x_n=\{x_n,x_1,\cdots,x_{n-1}\},
\end{equation}
 for all $x_i\in \mathcal{H}(\mathcal{A}),1\leq i\leq {n}$.\\
 
 $L$ is called left multiplication and $R$ is called right multiplication. 
If there is an $n$-Hom-pre-Lie superalgebra structure on its dual
space $\mathcal{A}^{*}$, we denote the left multiplication and right multiplication by $\mathcal{L}$ and $\mathcal{R}$ respectively.\\

 By the definitions of an $n$-Hom-pre-Lie superalgebra and a representation of an $n$-Hom-Lie superalgebra, we immediately obtain :
\begin{pro}
With the above notations, $(A,L,\alpha)$ is a representation of the
$n$-Hom-Lie superalgebra
$(\mathcal{A},[\cdot,\cdots,\cdot]^C,\alpha)$. On the other hand,
let $\mathcal{A}$ be a vector space with an  $n$-linear map
$\{\cdot,\cdots,\cdot\}:(\wedge^{n-1}\mathcal{A})\otimes \mathcal{A}\rightarrow \mathcal{A}$
. Then $(\mathcal{A},\{\cdot,\cdots,\cdot\},\alpha) $ is an $n$-Hom-pre-Lie superalgebra if $[\cdot,\cdots,\cdot]^C$ defined by Eq.~\eqref{eq:ncc} is an $n$-Hom-Lie superalgebra and the left multiplication $L$ defined by Eq.~\eqref{eq:R}
gives a representation of this $n$-Hom-Lie superalgebra.
\end{pro}
\begin{proof}
We skip the straightforward proof.
\end{proof}
\begin{pro}\label{pro:npreLieT}
Let $T:V\rightarrow A$ be an $\mathcal O$-operator on an $n$-Hom-Lie superalgebra $(\mathcal{A},[\cdot,\cdots,\cdot],\alpha)$ with respect to the representation $(V,\rho,\alpha_V)$ . Then there exists an $n$-Hom-pre-Lie superalgebra structure on $V$ given by
\begin{equation}\label{n-pre-Lie-O-operator}
\{u_1,\cdots,u_n\}_T=\rho(Tu_1,\cdots,Tu_{n-1})u_n,\quad\forall ~ u_i\in \mathcal{H}(V),1\leq i\leq n.
\end{equation}
In particular; If $V=\mathcal{A}$, let $P:\mathcal{A}\rightarrow \mathcal{A}$ be a Rota-Baxter operator of weight zero associated to $(\mathcal{A},ad)$. Then the compatible $n$-Hom-pre-Lie superalgebra on $\mathcal{A}$ is given by
\begin{equation}\label{eq:rott}
\{x_1,\cdots,x_n\}_P=[P(x_1),\cdots,P(x_{n-1}),x_n],
\end{equation}
for any $x_1,\cdots,x_n\in \mathcal{H}(\mathcal{A})$.
\end{pro}
\begin{proof}
Let $u_i,v_i\in\mathcal{H}(V),\;1\leq i\leq n$, then by using \eqref{repr-n-Hom-Lie1}, \eqref{eq:Ooperator} and \eqref{n-pre-Lie-O-operator}, we have:

{\small\begin{align*}
& \{\alpha_V(u_1),\cdots,\alpha_V(u_{n-1}),\{v_1,\cdots,v_n\}_T\}_T-(-1)^{|U|^{n-1}|V|^{n-1}}\{\alpha_V(v_1),\cdots,\alpha_V(v_{n-1}),\{u_1,\cdots,u_n\}_T\}_T\\&= \{\alpha_V(u_1),\cdots,\alpha_V(u_{n-1}),\rho(Tv_1,\cdots,Tv_{n-1})(v_n)\}_T-(-1)^{|U|^{n-1}|V|^{n-1}}\{\alpha_V(v_1),\cdots,\alpha_V(v_{n-1}),\rho(Tu_1,\cdots,Tu_{n-1})(u_n)\}_T\\&=\rho\big(T(\alpha_V(u_1)),\cdots,T(\alpha_V(u_{n-1}))\big)\rho\big(Tv_1,\cdots,Tv_{n-1}\big)(v_n)-(-1)^{|U|^{n-1}|V|^{n-1}}\rho\big(T(\alpha_V(v_1)),\cdots,T(\alpha_V(v_{n-1}))\big)\rho\big(Tu_1,\cdots,Tu_{n-1}\big)(u_n)\\&= \rho\big(\alpha(Tu_1),\cdots,\alpha(Tu_{n-1})\big)\rho\big(Tv_1,\cdots,Tv_{n-1}\big)(v_n)-(-1)^{|U|^{n-1}|V|^{n-1}}\rho\big(\alpha(Tv_1),\cdots,\alpha(Tv_{n-1})\big)\rho\big(Tu_1,\cdots,Tu_{n-1}\big)(u_n)\\&=\displaystyle\sum_{i=1}^{n-1}(-1)^{|U|^{n-1}|V|^{i-1}}\rho\big(\alpha(Tv_1),\cdots,\alpha(Tv_{i-1}),[Tu_1,\cdots,Tu_{n-1},Tv_i],\cdots,\alpha(Tv_{n-1})\big)\alpha_V\\&=\displaystyle\sum_{i=1}^{n-1}(-1)^{|U|^{n-1}|V|^{i-1}}\rho\Big(\alpha(Tv_1),\cdots,\alpha(Tv_{i-1}),\displaystyle T\Big(\sum_{j=1}^{n-1}(-1)^{n-i}(-1)^{|u_j||U|_{j+1}}\rho(Tu_1,\cdots,\hat{Tu_j},\cdots Tu_{n-1},Tv_i)u_j\Big),\cdots,\alpha(Tv_{n-1})\Big)\alpha_V\\&+ \displaystyle\sum_{i=1}^{n-1}(-1)^{|U|^{n-1}|V|^{i-1}}\rho\Big(\alpha(Tv_1),\cdots,\alpha(Tv_{i-1}),T(\rho(Tu_1,\cdots, Tu_{n-1})v_i),\cdots,\alpha(Tv_{n-1})\Big)\alpha_V\\&= \displaystyle\sum_{i=1}^{n-1}(-1)^{|U|^{n-1}|V|^{i-1}}\rho\Big(T(\alpha_V(v_1)),\cdots,T(\alpha_V(v_{i-1})),T([u_1,\cdots,u_{n-1},v_i]_V^C),\cdots,T(\alpha_V(v_{n-1}))\Big)\alpha_V\\&=\displaystyle\sum_{i=1}^{n-1}(-1)^{|U|^{n-1}|V|^{i-1}}\{\alpha_V(v_1),\cdots,\alpha_V(v_{i-1}),[u_1,\cdots,u_{n-1},v_i]_V^C,\cdots,\alpha_V(v_{n-1})\}_T,
\end{align*}}
which gives that the identity \eqref{n-Hom-pre-Lie-sup 1} is satisfied on $V$. By the same way we show that the identity \eqref{n-Hom-pre-Lie-sup 2} is satisfied. Then $(V,\rho,\alpha_V)$ is an $n$-Hom-pre-Lie superalgebra.\\
If $V=\mathcal{A}$, the result is obvious.
\end{proof}
\begin{cor}\label{n-Hom-pre-Lie T(V)}
With the above conditions,  $(V,[\cdot,\cdots,\cdot]^C,\alpha)$ is an $n$-Hom-Lie superalgebra as the sub-adjacent $n$-Hom-Lie superalgebra of the $n$-Hom-pre-Lie
superalgebra given in Proposition \ref{pro:npreLieT}, and $T$ is an $n$-Hom-Lie superalgebra morphism from $(V,[\cdot,\cdots,\cdot]^C,\alpha)$ to $(\mathcal{A},[\cdot,\cdots,\cdot],\alpha)$. Furthermore,
$T(V)=\{Tv\;|\;v\in V\}\subset A$ is an $n$-Hom-Lie subalgebra of $\mathcal{A}$ and there is an induced $n$-Hom-pre-Lie superalgebra structure $\{\cdot,\cdots,\cdot\}_{T(V)}$ on
$T(V)$ given by
\begin{equation}
\{Tu_1,\cdots,Tu_{n}\}_{T(V)}:=T\{u_1,\cdots,u_n\},\quad\;\forall u_i\in \mathcal{H}(V),1\leq i\leq n.
\end{equation}
\end{cor}

\begin{pro}\label{pro:preLieOoper}
Let $(\mathcal{A},[\cdot,\cdots,\cdot],\alpha)$ be an $n$-Hom-Lie superalgebra. Then there exists a compatible $n$-Hom-pre-Lie superalgebra if and only if there exists an invertible $\mathcal O$-operator $T:V\rightarrow \mathcal{A}$ with respect
to a representation $(V,\rho,\alpha_V)$. Furthermore, the compatible $n$-Hom-pre-Lie structure on $\mathcal{A}$ is given by
\begin{equation}
\{x_1,\cdots,x_n\}_{A}=T\rho(x_1,\cdots,x_{n-1})T^{-1}(x_n),\;\forall x_i\in \mathcal{H}(\mathcal{A}),1\leq i\leq n.
\end{equation}
\end{pro}

\begin{proof}
This is a direct computation, we apply Proposition \ref{pro:npreLieT} and corollary \ref{n-Hom-pre-Lie T(V)} for $T(V)=\mathcal{A}$.
\end{proof}

\subsection{ Representations of $n$-Hom-pre-Lie superalgebras}

In this subsection, we introduce the notion of a representation of an $n$-Hom-pre-Lie superalgebras which is the Hom-super case of \cite{Hajjaji-Chtioui-Mabrouk-Makhlouf}, so we give the construction of the corresponding semi-direct product $n$-Hom-pre-Lie superalgebra and we give some other results related this notion.  

\begin{defi}\label{defrep}
 Let $(\mathcal{A},\{\cdot,\cdots,\cdot\},\alpha)$ be an $n$-Hom-pre-Lie superalgebra. A  representation of $(\mathcal{A},\{\cdot,\cdots,\cdot\},\alpha)$ on a $\mathbb{Z}_2$-graded vector space $V$ is the given of a triple $(l,r,\alpha_V)$, where $l:\wedge^{n-1} \mathcal{A} \rightarrow gl(V)$ is a representation of the $n$-Hom-Lie superalgebra $\mathcal{A}^c$ on $V$,  $r:\mathcal{A}\times\cdots\times\mathcal{A} \rightarrow gl(V)$ is an even $(n-1)$-linear map super-skew-symmetric on the first $(n-2)$ terms  and $\alpha_V:V\to V$ is  an even linear maps such that  for all $x_1,\cdots,x_{n},y_1,\cdots,y_n\in \mathcal{H}(\mathcal{A})$, the following identities holds:
 \begin{align}
 &\bullet \alpha_Vr(x_1,\cdots,x_{n-1})=r(\alpha(x_1),\cdots,\alpha(x_{n-1}))\alpha_V,\label{cond-repres-r-1}\\
&\bullet l(\alpha(x_1),\cdots,\alpha(x_{n-1}))r(y_1,\cdots,y_{n-1}) =(-1)^{|X|^{n-1}|Y|^{n-1}}r(\alpha(y_1),\cdots,\alpha(y_{n-1}))\mu(x_1,\cdots,x_{n-1})\nonumber\\
&+\sum_{i=1}^{n-2}(-1)^{|X|^{n-1}|Y|^{i-1}}r(\alpha(y_1),\cdots,\alpha(y_{i-1}),[x_1,\cdots,x_{n-1},y_i]^C,\alpha(y_{i+1}),\cdots,\alpha(y_{n-1}))\alpha_V\label{cond-repres-l-r-1}\\ \nonumber&+(-1)^{|X|^{n-1}|Y|^{n-2}}r(\alpha(y_1),\cdots,\alpha(y_{n-2}),\{x_1,\cdots,x_{n-1},y_{n-1}\})\alpha_V, \\
&\bullet r([x_1,\cdots,x_n]^C,\alpha(y_{1}),\cdots,\alpha(y_{n-2}))\alpha_V=\sum_{i=1}^{n}(-1)^{n-i}(-1)^{|x_i||X|_{i+1}^n}l(\alpha(x_1),\cdots,\widehat{\alpha(x_i)},\cdots,\alpha(x_n))r(x_i,y_{1},\cdots,y_{n-2}),\label{cond-repres-l-r-2}\\
&\bullet r(\alpha(x_1),\cdots,\alpha(x_{n-2}),\{y_{1},\cdots,y_{n}\})\alpha_V=(-1)^{|X|^{n-2}|Y|^{n-1}}l(\alpha(y_{1}),\cdots,\alpha(y_{n-1}))r(x_1,\cdots,x_{n-2},y_{n})\nonumber\\
&+\sum_{i=1}^{n-1}(-1)^{i+1}(-1)^{(|X|^{n-2}+|y_i|)|Y|^{n}_{i+1}+|X|^{n-2}|Y|^{i-1}} r(\alpha(y_{1}),\cdots,\widehat{\alpha(y_i)},\cdots,\alpha(y_{n}))\mu(x_1,\cdots,x_{n-2},y_i),\label{cond-repres-l-r-3}\\
&\bullet r(\alpha(y_1),\cdots,\alpha(y_{n-1}))\mu(x_1,\cdots,x_{n-1})=(-1)^{|X|^{n-1}|Y|^{n-1}}l(\alpha(x_1),\cdots,\alpha(x_{n-1}))r(y_1,\cdots,y_{n-1})\nonumber\\
&+\sum_{i=1}^{n-1}(-1)^{i}(-1)^{|x_i||x|^{n-1}_{i+1}}r(\alpha(x_1),\cdots,\widehat{\alpha(x_i)},\cdots,\alpha(x_{n-1}),\{x_i,y_1,\cdots,y_{n-1}\})\alpha_V,\label{cond-repres-l-r-4}
\end{align}
where $\quad \mu(x_1,\cdots,x_{n-1})=l(x_1,\cdots,x_{n-1})+\displaystyle\sum_{i=1}^{n-1}(-1)^{i}(-1)^{|x_i||X|^{n-1}_{i+1}}r(x_1,\cdots,\widehat{x_i},\cdots,x_{n-1},x_i)$.
\end{defi} 

Let $(\mathcal{A},\{\cdot,\cdots,\cdot\},\alpha)$ be an $n$-Hom-pre-Lie superalgebra and $(l,\alpha_V)$ a representation of the sub-adjacent $n$-Hom-pre-Lie superalgebra
$\mathcal{A}^c$ on $V$ . Then $(l, r,\alpha_V)$ is a representation of the $n$-Hom-pre-Lie superalgebra
$(\mathcal{A},\{\cdot,\cdots,\cdot\},\alpha)$ on the $\mathbb{Z}_2$-graded vector space $V$.
It is obvious that $(\mathcal{A},L,R,\alpha)$ is a representation of an $n$-Hom-pre-Lie superalgebra on itself, which is called the adjoint representation.\\

\begin{thm}
Let $(V,l,r)$ be a representation of an $n$-pre-Lie superalgebra $(\mathcal{A},\{\cdot,\cdots,\cdot\})$. Let $\alpha_V\in gl(V)$ and $\alpha\in gl(\mathcal{A})$ two morphisms such that 
$$\alpha_Vl(x_1,\cdots,x_{n-1})=l(\alpha(x_1),\cdots,\alpha(x_{n-1}))\alpha_V,\;\;\;\alpha_Vr(x_1,\cdots,x_{n-1})=r(\alpha(x_1),\cdots,\alpha(x_{n-1}))\alpha_V,$$ 
for all $x_i\in\mathcal{H}(\mathcal{A}),\;1\leq i\leq n-1$.
Then $(V,\widetilde{l},\widetilde{r},\alpha_V)$ is a representation on the $n$-Hom-pre-Lie superalgebras $(\mathcal{A},\{\cdot,\cdots,\cdot\}_\alpha^C,\alpha)$, where $\widetilde{l}=\alpha_V\circ l$, $\widetilde{r}=\alpha_V\circ r$ and $\{\cdot,\cdots,\cdot\}_\alpha^C=\alpha\circ\{\cdot,\cdots,\cdot\}^C$.
\end{thm}

\begin{proof}
Let $x_i,y_i\in\mathcal{H}(\mathcal{A}),\;1\leq i\leq n$, then by condition \eqref{cond-repres-r-1}, we have
\begin{align*}
\alpha_V\widetilde{r}(x_1,\cdots,x_{n-1})&=\alpha_V\alpha_Vr(x_1,\cdots,x_{n-1})=\alpha_Vr(\alpha(x_1),\cdots,\alpha(x_{n-1}))\alpha_V\\&=\widetilde{r}(\alpha(x_1),\cdots,\alpha(x_{n-1}))\alpha_V.   
\end{align*}
Then, the condition \eqref{cond-repres-r-1} is satisfied by $\widetilde{r}$. By the same way, we show that the conditions  \eqref{cond-repres-l-r-1}-\eqref{cond-repres-l-r-4} hold. The theorem is proved. 
\end{proof}

\begin{pro}\label{carpre}
Let $(\mathcal{A},\{\cdot,\cdots,\cdot\},\alpha)$ be an $n$-Hom-pre-Lie superalgebra, $V$  a $\mathbb{Z}_2$-graded vector space and $l,r:
\otimes^{n-1}\mathcal{A}\rightarrow  gl(V)$  two even linear
maps. Then $(V,l,r,\alpha_V)$ is a representation of $\mathcal{A}$ if and only if there
is an $n$-Hom-pre-Lie superalgebra structure $($called semi-direct product$)$
on the direct sum $\mathcal{A}\oplus V$ of vector spaces, defined by
\begin{align}
\{x_1+u_1,\cdots,x_n+u_n\}_{\mathcal{A}\oplus V}=&\{x_1,\cdots,x_n\}+l(x_1,\cdots,x_{n-1})(u_n)\nonumber
\\ \label{eq:sum}&+\sum_{i=1}^{n-1}(-1)^{i+1}(-1)^{|x_i||X|_{i+1}^n}r(x_1,\cdots,\widehat{x_i},\cdots,x_n)(u_i),
\end{align}
for $x_i\in\mathcal{H}(\mathcal{A}), u_i\in\mathcal{H}(V), 1\leq i\leq n$. We denote this semi-direct product $n$-Hom-pre-Lie superalgebra by $\mathcal{A}\ltimes_{l,r}^{\alpha_V} V.$
\end{pro}
\begin{proof}
Let $x_i\in\mathcal{H}(\mathcal{A}),\;u_i\in\mathcal{H}(V),\;1\leq i\leq n$, then, for all $1\leq j\leq n-2$, we have

{\small\begin{align*}
& \{x_1+u_1,\cdots,x_j+u_j,x_{j+1}+u_{j+1},\cdots,x_n+u_n\}_{\mathcal{A}\oplus V}\;\;\\&=\{x_1,\cdots,x_j,x_{j+1},\cdots,x_n\}+l(x_1,\cdots,x_j,x_{j+1},\cdots,x_{n-1})(u_n)+\displaystyle\sum_{i=1}^{j-1}(-1)^{i+1}(-1)^{|x_i||X|^{n-1}_{i+1}}r(x_1,\cdots,\hat{x_i},\cdots,x_j,x_{j+1},\cdots,x_n)(u_i)\\& +\displaystyle\sum_{i=j+2}^{n-1}(-1)^{i+1}(-1)^{|x_i||X|^{n-1}_{i+1}}r(x_1,\cdots,x_j,x_{j+1},\hat{x_i},\cdots,\cdots,x_n)(u_i)+(-1)^{j+1}(-1)^{|x_j||X|^{n-1}_{j+1}}r(x_1,\cdots,x_{j-1},x_{j+1},\cdots,x_n)(u_j)\\&+(-1)^{j}(-1)^{|x_{j+1}||X|^{n-1}_{j+2}}r(x_1,\cdots,x_j,x_{j+2},\cdots,x_n)(u_{j+1})\\&=-(-1)^{|x_j||x_{j+1}|}\Big(\{x_1,\cdots,x_{j+1},x_j,\cdots,x_n\}+l(x_1,\cdots,x_{j+1},x_j,\cdots,x_{n-1})(u_n)\\&+\displaystyle\sum_{i=1}^{j-1}(-1)^{i+1}(-1)^{|x_i||X|^{n-1}_{i+1}}r(x_1,\cdots,\hat{x_i},\cdots,x_{j+1},x_j,\cdots,x_n)(u_i)+\displaystyle\sum_{i=j+2}^{n-1}(-1)^{i+1}(-1)^{|x_i||X|^{n-1}_{i+1}}r(x_1,\cdots,x_{j+1},x_j,\hat{x_i},\cdots,\cdots,x_n)(u_i)\\&+(-1)^j(-1)^{|x_j||X|^{n-1}_{j+2}}r(x_1,\cdots,x_{j-1},x_{j+1},\cdots,x_n)(u_j)+(-1)^{j+1}(-1)^{|x_{j+1}|(|x_j|+|X|^{n-1}_{j+1})}r(x_1,\cdots,x_j,x_{j+2},\cdots,x_n)(u_{j+1}) \Big)\\&=(-1)^{(|x_j+uj|)(|x_{j+1}+u_{j+1}|)}\{x_1+u_1,\cdots,x_{j+1}+u_{j+1},x_j+u_j,\cdots,x_n+u_n\}_{\mathcal{A}\oplus V}, 
\end{align*}}
which implies that $\{\cdot,\cdots,\cdot\}_{\mathcal{A}\oplus V}$ is super-skew-symmetric on the first $(n-1)$ terms. 
\end{proof}
Let $V$ be a $\mathbb{Z}_2$-graded vector space and $(V,l,r,\alpha_V)$ be a representation of the $n$-Hom-pre-Lie superalgebra $(\mathcal{A},\{\cdot,\cdots,\cdot\},\alpha)$ on $V$. Define $\widetilde{\rho}:\wedge^{n-1}\mathcal{A} \rightarrow gl(V)$ by
\begin{equation}
    \widetilde{\rho}(x_1,\cdots,x_{n-1})= l(x_1,\cdots,x_{n-1})+\sum_{i=1}^{n-1}(-1)^{i}(-1)^{|x_i||X|_{i+1}^{n-1}}r(x_1,\cdots,\widehat{x_i},\cdots,x_{n-1},x_i),
\end{equation}
for all $x_1,\cdots,x_{n-1}\in\mathcal{H}(\mathcal{A}).$
\begin{pro}\label{teald}
With the above notation, $(V,\widetilde{\rho},\alpha_V)$ is a representation of the sub-adjacent $n$-Hom-Lie superalgebra $(\mathcal{A}^c, [\cdot,\cdots, \cdot]^C,\alpha)$ on the $\mathbb{Z}_2$-graded  vector space $V$.
\end{pro}
\begin{proof}
By Proposition \ref{carpre}, we have the semi-direct product $n$-Hom-pre-Lie superalgebra $\mathcal{A}\ltimes^{\alpha_V}_{l,r} V.$ Consider its sub-adjacent $n$-Hom-Lie superalgebra structure
$[\cdot,\cdots, \cdot]^C$, we have for any $x_i\in\mathcal{H}(\mathcal{A}),\;u_i\in\mathcal{H}(V)$
\begin{align}
  &\quad \;[x_1+u_1,\cdots,x_n+u_n]_{\mathcal{A}\oplus V}^C=\sum_{i=1}^{n}(-1)^{n-i}(-1)^{|x_i||X|^n_{i+1}}\{x_1+u_1,\cdots,\widehat{x_i+u_i},\cdots,x_n+u_n,x_i+u_i\}_{\mathcal{A}\oplus V}\nonumber\\
  &=\sum_{i=1}^{n}(-1)^{n-i}(-1)^{|x_i||X|^n_{i+1}}\{x_1,\cdots,\widehat{x_i},\cdots,x_n,x_i\}+\sum_{i=1}^{n}(-1)^{n-i}(-1)^{|x_i||X|^n_{i+1}}l(x_1,\cdots,\widehat{x_i},\cdots,x_n)(u_i)\nonumber\\
  &+\sum_{i=1}^{n}(-1)^{n-i}\Big(\sum_{1 \leq i<j \leq n}(-1)^{j}(-1)^{|X|^n_{j+1}(|x_i|+|x_j|)+|x_i||X|^{j-1}_{i+1}+|x_i||x_j|}r(x_1,\cdots,\widehat{x_i},\cdots,\widehat{x_j},\cdots,x_n,x_i)(u_j)\nonumber\\
  &+\sum_{1 \leq j<i \leq n}(-1)^{j+1}(-1)^{|X|^n_{j+1}(|x_i|+|x_j|)+|x_j||X|^{i-1}_{j+1}+|x_i||x_j|}r(x_1,\cdots,\widehat{x_j},\cdots,\widehat{x_i},\cdots,x_n,x_i)(u_j)\Big)\nonumber\\
  &=[x_1,\cdots,x_n]^C+\sum_{i=1}^{n}(-1)^{n-i}(-1)^{|x_i||X|^{n}_{i+1}}\Big(l(x_1,\cdots,\widehat{x_i},\cdots,x_n)(u_i)\\&+\sum_{1 \leq i<j \leq n}(-1)^{j}(-1)^{j}(-1)^{|X|^n_{j+1}(|x_i|+|x_j|)+|x_i||X|^{j-1}_{i+1}+|x_i||x_j|}r(x_1,\cdots,\widehat{x_i},\cdots,\widehat{x_j},\cdots,x_n,x_i)(u_j)\nonumber\\
  &+\sum_{1 \leq j<i \leq n}(-1)^{j+1}(-1)^{|X|^n_{j+1}(|x_i|+|x_j|)+|x_j||X|^{i-1}_{j+1}+|x_i||x_j|}r(x_1,\cdots,\widehat{x_j},\cdots,\widehat{x_i},\cdots,x_n,x_i)(u_j)\nonumber\\
  &=[x_1,\cdots,x_n]^C+\sum_{k=1}^{n}(-1)^{n-k}(-1)^{|x_k||X|^{n}_{k+1}}\;\widetilde{\rho}(x_1,\cdots,\widehat{x_k},\cdots,x_n)(u_k)\label{brac}.
\end{align}
By Proposition \ref{direct-sum-n-hom-lie-superalgebras}, $(V,\widetilde{\rho},\alpha_V)$ is a representation of the sub-adjacent $n$-Hom-Lie superalgebra $(A^c, [\cdot,\cdots, \cdot]^C,\alpha)$ on the $\mathbb{Z}_2$-vector space $V$. The proof is finished.
\end{proof}

If $(l,r,\alpha_V)=(L,R,\alpha_V)$ is a representation of an $n$-Hom-pre-Lie superalgebra $(\mathcal{A},\{\cdot,\cdots,\cdot\},\alpha)$, then $\widetilde{\rho}=ad$ is the adjoint representation of the sub-adjacent $n$-Hom-Lie superalgebra $(\mathcal{A}^c, [\cdot,\cdots, \cdot]^C,\alpha)$ on itself.
\begin{cor} Let
$(V,l,r,\alpha_V)$ be a representation of an $n$-Hom-pre-Lie superalgebra $(\mathcal{A},\{\cdot,\cdots,\cdot\},\alpha)$ on $V$. Then the semi-product $n$-Hom-pre-Lie superalgebras $\mathcal{A}\ltimes^{\alpha_V}_{l,r}V$ and $\mathcal{A}\ltimes^{\alpha_V}_{\widetilde{\rho}}V$ given by the representations $(V,l,r,\alpha_V)$ and $(V,\widetilde{\rho},0,\alpha_V)$ respectively have the same sub-adjacent $n$-Hom-Lie superalgebra $\mathcal{A}^c\ltimes^{\alpha_V}_{\widetilde{\rho}}V$ given by \eqref{brac}.
\end{cor}

Let $(V,l,r,\alpha_V)$ be a representation of an $n$-Hom-pre-Lie superalgebra $(\mathcal{A},\{\cdot,\cdots,\cdot\},\alpha)$. In the sequel, we always assume that $\alpha_V$ is invertible to study the dual representation. For all $x_1,\cdots,x_{n-1}\in\mathcal{H}(\mathcal{A}),\;u\in\mathcal{H}(V),\;\xi\in V^*$, define $\widetilde{\rho}^*,r^*:\otimes^{n-1}\mathcal{A}\to gl(V^*)$  by
$$<\widetilde{\rho}^*(x_1,\cdots,x_{n-1})(\xi),u>=-<\xi,\widetilde{\rho}(x_1,\cdots,x_{n-1})(u)>,$$ and 
$$<r^*(x_1,\cdots,x_{n-1})(\xi),u>=-<\xi,r(x_1,\cdots,x_{n-1})(u)>.$$
Then, define $\widetilde{\rho}^\star,r^\star:\otimes^{n-1}\mathcal{A}\to gl(V^*)$ by 
\begin{equation}\label{star-right-repr1}
   \widetilde{\rho}^\star(x_1,\cdots,x_{n-1})(\xi):=\widetilde{\rho}^*(\alpha(x_1),\cdots,\alpha(x_{n-1}))((\alpha_V^{-2})^*(\xi)), 
\end{equation}
\begin{equation}\label{star-right-repr2}
   r^\star(x_1,\cdots,x_{n-1})(\xi):=r^*(\alpha(x_1),\cdots,\alpha(x_{n-1}))((\alpha_V^{-2})^*(\xi)). 
\end{equation}
\begin{thm}\label{adjoint-rep-n-hom-pre-lie}
Let $(V,l,r,\alpha_V)$ be a representation of an $n$-Hom-pre-Lie superalgebra $(\mathcal{A},\{\cdot,\cdots,\cdot\},\alpha)$ on $V$ where $\alpha_V$ is invertible. Then
$(V^{*},\widetilde{\rho}^{\star},-r^\star,(\alpha_V^{-1})^*)$ is a representation of the $n$-Hom-pre-Lie superalgebra $(\mathcal{A},\{\cdot,\cdots,\cdot\},\alpha)$ on $V^{*}$, which is called the dual representation of the representation $(V,l,r,\alpha_V)$.
\end{thm}
\begin{proof}
By Proposition \ref{teald}, $(V,\widetilde{\rho},\alpha_V)$ is a representation of the sub-adjacent $n$-Hom-Lie superalgebra $(\mathcal{A}^c, [\cdot,\cdots, \cdot]^C,\alpha)$ on $V$. By Example \ref{dual-representation}, $(V^{*},\widetilde{\rho}^{\star},(\alpha_V^{-1})^*)$ is a representation of the sub-adjacent $n$-Lie algebra $(\mathcal{A}^c, [\cdot,\cdots, \cdot]^C,\alpha)$ on the dual vector space $V^{*}$. It is straightforward to deduce that other conditions of Definition \ref{defrep} also holds. We leave details to readers.
\end{proof}

The tensor product of two representations of an $n$-Hom-pre-Lie superalgebras is still a representation.
\begin{thm}
Let $(\mathcal{A},\{\cdot,\cdots,\cdot\},\alpha)$ is an $n$-Hom-pre-Lie superalgebra, $(V_1,l_{V_1},r_{V_1},\alpha_{V_1})$ and $(V_2,l_{V_2},r_{V_2},\alpha_{V_2})$ its representations. Then $(V_1\otimes V_2,l_{V_1}\otimes\alpha_{V_2}+\alpha_{V_1}\otimes(l_{V_2}-r_{V_2}),r_{V_1}\otimes\alpha_{V_2},\alpha_{V_1}\otimes\alpha_{V_2})$ is a representation of $(\mathcal{A},\{\cdot,\cdots,\cdot\},\alpha)$.
\end{thm}
\begin{proof}
By using the fact that $(V_1,l_{V_1},r_{V_1},\alpha_{V_1})$ and $(V_2,l_{V_2},r_{V_2},\alpha_{V_2})$ are representations of $(\mathcal{A},\{\cdot,\cdots,\cdot\},\alpha)$, then for all $x_1,\cdots,x_{n},y_1,\cdots,y_{n-1}\in \mathcal{H}(\mathcal{A})$, we have:
 \begin{align}
 \bullet &(\alpha_{V_1}\otimes \alpha_{V_2})\circ (r_{V_1}\otimes\alpha_{V_2})(x_1,\cdots,x_{n-1})-(r_{V_1}\otimes \alpha_{V_2})\circ(\alpha_{V_1}\otimes \alpha_{V_2})(\alpha(x_1),\cdots,\alpha(x_{n-1}))\nonumber\\
 &=(\alpha_{V_1}\circ r_{V_1}(x_1,\cdots,x_{n}))\otimes(\alpha_{V_2}\otimes \alpha_{V_2})-(r_{V_1}(\alpha(x_1),\cdots,\alpha(x_{n-1}))\circ\alpha_{V_1})\otimes(\alpha_{V_2}\circ\alpha_{V_2})\nonumber\\&=0\nonumber\\
 \bullet &(l_{V_1}\otimes\alpha_{V_2}+\alpha_{V_1}\otimes(l_{V_2}-r_{V_2}))(\alpha(x_1),\cdots,\alpha(x_{n-1}))\circ (r_{V_1}\otimes \alpha_{V_2})(y_1,\cdots,y_{n-1})\nonumber\\ &-(-1)^{|X|^{n-1}|Y|^{n-1}}(r_{V_1}\otimes \alpha_{V_2})(\alpha(y_1),\cdots,\alpha(y_{n-1}))\circ\mu_{V_1\otimes V_2}(x_1,\cdots,x_{n-1})\nonumber\\
 &+\sum_{i=1}^{n-2}(-1)^{|X|^{n-1}|Y|^{i-1}}(r_{V_1}\otimes \alpha_{V_2})(\alpha(y_1),\cdots,\alpha(y_{i-1}),[x_1,\cdots,x_{n-1},y_i]^C,\alpha(y_{i+1}),\cdots,\alpha(y_{n-1}))\circ(\alpha_{V_1}\otimes\alpha_{V_2}) \nonumber\\&+(-1)^{|X|^{n-1}|Y|^{n-2}}(r_{V_1}\otimes \alpha_{V_2})(\alpha(y_1),\cdots,\alpha(y_{n-2}),\{x_1,\cdots,x_{n-1},y_{n-1}\})\circ(\alpha_{V_1}\otimes\alpha_{V_2})\nonumber\\&=(l_{V_1}(\alpha(x_1),\cdots,\alpha(x_{n-1}))\circ r_{V_1}(y_1,\cdots,y_{n-1}))\otimes(\alpha_{V_2}\circ\alpha_{V_2})\nonumber\\
 &+((\alpha_{V_1}\circ r_{V_1})(y_1,\cdots,y_{n-1}))\otimes(l_{V_2}(\alpha(x_1),\cdots,\alpha(x_{n-1}))\circ \alpha_{V_2
 }\nonumber\\
 &-((\alpha_{V_1}\circ r_{V_1})(y_1,\cdots,y_{n-1}))\otimes (r_{V_2}(\alpha(x_1),\cdots,\alpha(x_{n-1}))\circ \alpha_{V_2
 }\nonumber\\
 &-(-1)^{|X|^{n-1}|Y|^{n-1}}(r_{V_1}\otimes \alpha_{V_2})(\alpha(y_1),\cdots,\alpha(y_{n-2})\circ ((l_{V_1}\otimes \alpha_{V_2}\nonumber+\alpha_{V_1}\otimes(l_{V_2}-r_{V_2}))(x_1,\cdots,x_{n-1})\\&+\sum_{i=1}^{n-1}(-1)^{i}(-1)^{|X_i||X|_{i+1}^{n-1}}(r_{V_1}\otimes \alpha_{V_2})(x_1,\cdots,\hat{x_i},\cdots,x_{n-1},x_i))\nonumber\\
 &+\sum_{i=1}^{n-2}(-1)^{|X|^{n-1}|Y|^{i-1}}(r_{V_1}(\alpha(y_1),\cdots,\alpha(y_{i-1}),[x_1,\cdots,x_{n-1},y_i]^C,\alpha(y_{i+1}),\cdots,\alpha(y_{n-1}))\circ \alpha_{V_1})\otimes(\alpha_{V_2}\circ\alpha_{V_2} ) \nonumber\\
 &+(-1)^{|X|^{n-1}|Y|^{n-2}}(r_{V_1}(\alpha(y_1),\cdots,\alpha(y_{n-2}),\{x_1,\cdots,x_{n-1},y_{n-1}\})\circ\alpha_{V_1})\otimes(\alpha_{V_2}\circ\alpha_{V_2})\nonumber\\&=(l_{V_1}(\alpha(x_1),\cdots,\alpha(x_{n-1}))\circ r_{V_1}(y_1,\cdots,y_{n-1}))\otimes(\alpha_{V_2}\circ\alpha_{V_2})\nonumber\\
 &+((\alpha_{V_1}\circ r_{V_1})(y_1,\cdots,y_{n-1}))\otimes(l_{V_2}(\alpha(x_1),\cdots,\alpha(x_{n-1}))\circ \alpha_{V_2
 }\nonumber\\&-((\alpha_{V_1}\circ r_{V_1})(y_1,\cdots,y_{n-1}))\otimes (r_{V_2}(\alpha(x_1),\cdots,\alpha(x_{n-1}))\circ \alpha_{V_2
 }\nonumber\\ &-(-1)^{|X|^{n-1}|Y|^{n-1}}(r_{V_1}(\alpha(y_1),\cdots,\alpha(y_{n-1})\circ (l_{V_1}(x_1,\cdots,x_{n-1}))\otimes( \alpha_{V_2}\circ\alpha_{V_2})\nonumber\\&-(-1)^{|X|^{n-1}|Y|^{n-1}}(r_{V_1}(\alpha(y_1),\cdots,\alpha(y_{n-1})\circ \alpha_{V_1}))\otimes (\alpha_{V_2}\circ l_{V_2}(x_1,\cdots,x_{n-1})\nonumber\\&+(-1)^{|X|^{n-1}|Y|^{n-1}}(r_{V_1}(\alpha(y_1),\cdots,\alpha(y_{n-1})\circ \alpha_{V_1}))\otimes (\alpha_{V_2}\circ r_{V_2}(x_1,\cdots,x_{n-1})\nonumber\\&-(-1)^{|X|^{n-1}|Y|^{n-1}}\sum_{i=1}^{n-1}(-1)^{i}(-1)^{|X_i||X|_{i+1}^{n-1}}(r_{V_1}(\alpha(y_1),\cdots,\alpha(y_{n-1})\circ r_{V_1}(x_1,\cdots,\hat{x_i},\cdots,x_{n-1},x_i))\otimes(\alpha_{V_2}\circ\alpha_{V_2} ) \nonumber\\
 &+\sum_{i=1}^{n-2}(-1)^{|X|^{n-1}|Y|^{i-1}}(r_{V_1}(\alpha(y_1),\cdots,\alpha(y_{i-1}),[x_1,\cdots,x_{n-1},y_i]^C,\alpha(y_{i+1}),\cdots,\alpha(y_{n-1}))\circ \alpha_{V_1})\otimes(\alpha_{V_2}\circ\alpha_{V_2} ) \nonumber\\
 &+(-1)^{|X|^{n-1}|Y|^{n-2}}(r_{V_1}(\alpha(y_1),\cdots,\alpha(y_{n-2}),\{x_1,\cdots,x_{n-1},y_{n-1}\})\circ\alpha_{V_1})\otimes(\alpha_{V_2}\circ\alpha_{V_2})\nonumber\\
 &=0\nonumber\\
 &\bullet (r_{V_1}\otimes \alpha_{V_2})([x_1,\cdots,x_n]^C,\alpha(y_{1}),\cdots,\alpha(y_{n-2}))\circ(\alpha_{V_1}\otimes \alpha_{V_2})\nonumber\\
 &-\sum_{i=1}^{n}(-1)^{n-i}(-1)^{|x_i||X|_{i+1}^n}(l_{V_1}\otimes\alpha_{V_2}+\alpha_{V_1}\otimes(l_{V_2}-r_{V_2}))(\alpha(x_1),\cdots,\widehat{\alpha(x_i)}),\cdots,\alpha(x_n))\circ (r_{V_1}\otimes\alpha_{V_2})(x_i,y_1,\cdots,y_{n-2})\nonumber\\&=(r_{V_1}([x_1,\cdots,x_n]^C,\alpha(y_{1}),\cdots,\alpha(y_{n-2}))\circ\alpha_{V_1})\otimes(\alpha_{V_2}\circ\alpha_{V_2})\nonumber\\&-\sum_{i=1}^{n}(-1)^{n-i}(-1)^{|x_i||X|_{i+1}^n}(l_{V_1}(\alpha(x_1),\cdots,\widehat{\alpha(x_i}),\cdots,\alpha(x_n))\circ (r_{V_1}(x_i,y_1,\cdots,y_{n-2}))\otimes(\alpha_{V_2}\circ \alpha_{V_2})\nonumber\\&-\sum_{i=1}^{n}(-1)^{n-i}(-1)^{|x_i||X|_{i+1}^n}(\alpha_{V_1}\circ r_{V_1}(x_i,y_1,\cdots,y_{n-2}))\otimes(l_{V_2}(\alpha(x_1),\cdots,\widehat{\alpha(x_i}),\cdots,\alpha(x_n))\circ \alpha_{V_2})\nonumber\\&+\sum_{i=1}^{n}(-1)^{n-i}(-1)^{|x_i||X|_{i+1}^n}(\alpha_{V_1}\circ r_{V_1}(x_i,y_1,\cdots,y_{n-2}))\otimes(r_{V_2}(\alpha(x_1),\cdots,\widehat{\alpha(x_i}),\cdots,\alpha(x_n))\circ \alpha_{V_2})\nonumber\\&=0\nonumber.
\end{align}
By the same way, we show that the conditions (\ref{cond-repres-l-r-3}) and (\ref{cond-repres-l-r-4}) hold. The proposition is proved.
\end{proof}
\begin{lem}\label{adjoint-adjoit-rep}
Let $(V,l,r,\alpha_V)$ be a representation of an $n$-Hom-pre-Lie superalgebra $(\mathcal{A},\{\cdot,\cdots,\cdot\},\alpha)$.  Then we have $$(\widetilde{\rho}^\star)^\star=\widetilde{\rho}\;\;\text{and}\;\;(r^\star)^\star=r,$$ where $\widetilde{\rho}^\star$ and $r^\star$ defined by \eqref{star-right-repr1} and \eqref{star-right-repr2} respectively.
\end{lem}
\begin{proof}
Let $x_i\in\mathcal{H}(\mathcal{A}),\;1\leq i\leq n-1,\;\xi\in V^\ast$ and $u\in V$, then we have
\begin{align*}
<(\widetilde{\rho}^\star)^\star(x_1,\cdots,x_{n-1})(u),\xi>&=<(\widetilde{\rho}^\star)^\ast(\alpha(x_1),\cdots,\alpha(x_{n-1}))(\alpha_V^2(u)),\xi>\\&=- <\alpha_V^2(u),\widetilde{\rho}^\star(\alpha(x_1),\cdots,\alpha(x_{n-1}))(\xi)>\\&=-<\alpha_V^2(u),\widetilde{\rho}^\ast(\alpha^2(x_1),\cdots,\alpha^2(x_{n-1}))((\alpha_V^{-2})^\ast(\xi))>\\&= <\widetilde{\rho}(\alpha^2(x_1),\cdots,\alpha^2(x_{n-1}))(\alpha_V^2(u)),((\alpha_V^{-2})^\ast(\xi))>\\&= <\widetilde{\rho}(x_1,\cdots,x_{n-1})(u),\xi>, 
\end{align*}
Then $(\widetilde{\rho}^\star)^\star=\widetilde{\rho}$. Similarly, we have $(r^\star)^\star=r$.
\end{proof}
\begin{pro}\label{adjoit-adjoint-rep}
Let $(V,l,r,\alpha_V)$ be a representation of an $n$-Hom-pre-Lie superalgebra $(\mathcal{A},\{\cdot,\cdots,\cdot\},\alpha)$, where $\alpha_V$ is invertible. Then the dual representation of $(V^*,\widetilde{\rho}^\star,-r^\star,(\alpha_V^{-1})^*)$ is $(V,\widetilde{\rho},-r,\alpha_V)$.
\end{pro}
\begin{proof}
It is obviously that $(V^\ast)^\ast=V$ and $(((\alpha_V^{-1})^\ast)^{-1})^\ast=\alpha_V$. Using also Lemma \ref{adjoint-adjoit-rep}, we obtain the result.
\end{proof}

\begin{pro}
Let $(V,l,r,\alpha_V)$ be a representation of an $n$-Hom-pre-Lie superalgebra $(\mathcal{A},\{\cdot,\cdots,\cdot\},\alpha)$, where $\alpha_V$ is invertible. Then the following conditions are equivalent:
\begin{enumerate}
    \item The quadruple $(V,\widetilde{\rho},-r,\alpha_V)$ is a representation of the $n$-Hom-pre-Lie superalgebra $(\mathcal{A},\{\cdot,\cdots,\cdot\},\alpha)$.
    \item The quadruple $(V^\ast,\widetilde{\rho}^\star+r^\star,r^\star,(\alpha_V^{-1})^\ast)$ is a representation of the $n$-Hom-pre-Lie superalgebra $(\mathcal{A},\{\cdot,\cdots,\cdot\},\alpha)$.
    \item $r(\alpha(x_1),\cdots,\alpha(x_{n-1}))r(y_1,\cdots,y_{n-1})=-(-1)^{|X|^{n-1}|Y|^{n-1}}r(\alpha(y_1),\cdots,\alpha(y_{n-1}))r(x_1,\cdots,x_{n-1}),$ for all $x_i,y_i\in\mathcal{H}(\mathcal{A}),\;1\leq i\leq n-1. $
\end{enumerate}
\end{pro}

\begin{proof}
The equivalence of conditions $1$ and $2$ is deduced directly from the Theorem \ref{adjoint-rep-n-hom-pre-lie} and Proposition \ref{adjoit-adjoint-rep}. By using the fact that $(V,l,r,\alpha_V)$ is a representation of an $n$-Hom-pre-Lie superalgebra $(\mathcal{A},\{\cdot,\cdots,\cdot\},\alpha)$, we deduce that conditions $1$ and $3$ are equivalent. 
\end{proof}
\begin{cor}
Let $(\mathcal{A},\{\cdot,\cdots,\cdot\},\alpha)$ be a regular $n$-Hom-pre-Lie superalgebra. Then $(\mathcal{A}^\ast,ad^\star,-R^\star,(\alpha_V^{-1})^\ast)$ is a representation of $(\mathcal{A},\{\cdot,\cdots,\cdot\},\alpha)$. 
\end{cor}

\begin{defi}
Let $(\mathcal{A},\{\cdot,\cdots,\cdot\},\alpha)$ be an $n$-Hom-pre-Lie superalgebra and $(V,l,r,\alpha_V)$ be a representation. An even linear map $T:V \rightarrow \mathcal{A}$ is called an $\mathcal O$-operator associated to $(V,l,r,\alpha_V)$ if $T$ satisfies
\begin{equation}\label{cond-O-operator}
    \alpha\circ T=T\circ\alpha_V,
\end{equation}
\begin{equation}\label{O-op n-pre-Lie}
    \{Tu_1,\cdots,Tu_n\}=T\Big(l(Tu_1,\cdots,Tu_{n-1})(u_n)+\sum_{i=1}^{n-1}(-1)^{i+1}(-1)^{|u_i||U|^n_{i+1}}r(Tu_1,\cdots,\widehat{Tu_i},\cdots,Tu_n)(u_i)\Big),
\end{equation}
$\forall u_i\in \mathcal{H}(V),\;1\leq i\leq n$. If $(V,l,r,\alpha_V)=(\mathcal{A},L,R,\alpha)$, then $T$ is called a Rota-Baxter operator on $\mathcal{A}$ of weight zero denoted by $P$.
\end{defi}

\begin{pro}\label{n-Hom-pre-Lie-Rota-B}
Let $P$ be a Rota-Baxter operator of weight zero on an $n$-Hom-pre-Lie superalgebra $(\mathcal{A},\{\cdot,\cdots,\cdot\},\alpha)$. Then $(\mathcal{A},\{\cdot,\cdots,\cdot\}_P,\alpha)$ is an $n$-Hom-pre-Lie superalgebra where $\{\cdot,\cdots,\cdot\}_P$ is defined by 
$$\{x_1,\cdots,x_n\}_P=\displaystyle\sum_{i=1}^n\{P(x_1),\cdots,P(x_{i-1}),x_i,P(x_{i+1}),\cdots,P(x_n)\},$$
for all $x_i\in\mathcal{H}(\mathcal{A}),\;1\leq i\leq n$.
\end{pro}
To show this proposition we need the following lemma.
\begin{lem}
Let $P$ be a Rota-Baxter operator of weight zero on an $n$-Hom-pre-Lie superalgebra $(\mathcal{A},\{\cdot,\cdots,\cdot\},\alpha)$. Then $P$ is a Rota-Baxter operator of weight zero on the sub-adjacent n-Hom-Lie superalgebra $\mathcal{A}^C$.
\end{lem}
\begin{proof}
It is obvious that $\alpha\circ P=P\circ\alpha$. \\
Let $x_i\in\mathcal{H}(\mathcal{A}),\;1\leq i\leq n$, we have
\begin{align*}
[P(x_1),\cdots,P(x_n)]^C&=\displaystyle\sum_{i=1}^n(-1)^{n-i}(-1)^{|x_i||X|^n_{i+1}}\{P(x_1),\cdots,\hat{P(x_i)},\cdots,P(x_n),P(x_i)\}\\&=\displaystyle\sum_{i=1}^n(-1)^{n-i}(-1)^{|x_i||X|^n_{i+1}}P\Big(\displaystyle\sum_{j=1}^n\{P(x_1),\cdots,x_j,\cdots,\hat{P(x_i)},\cdots,P(x_n),P(x_i)\}\Big)\\&=P\Big(\displaystyle\sum_{j=1}^n\displaystyle\sum_{i=1}^n(-1)^{n-i}(-1)^{|x_i||X|^n_{i+1}}\{P(x_1),\cdots,x_j,\cdots,\hat{P(x_i)},\cdots,P(x_n),P(x_i)\}\Big)\\&=P\Big(\displaystyle\sum_{j=1}^n[P(x_1),\cdots,x_j,\cdots,P(x_n)]\Big).    
\end{align*}
Then $P$ is a Rota-Baxter operator on $(\mathcal{A}^C,[\cdot,\cdots,\cdot]^C,\alpha)$
\end{proof}

\begin{rmk}\label{p-crochet-p}
$P(\{x_1,\cdots,x_n\}_p)=\{P(x_1),\cdots,P(x_{n})\},\;\forall x_i\in\mathcal{H}(\mathcal{A}),\;1\leq i\leq n$.
\end{rmk}

\begin{proof}[Proof of Proposition \ref{n-Hom-pre-Lie-Rota-B}]
Let $x_i,y_i\in\mathcal{H}(\mathcal{A}),\;1\leq i\leq n-1$. By using Remark \ref{p-crochet-p} and condition \eqref{n-Hom-pre-Lie-sup 1}, we have
{\small\begin{align*}
 \{\alpha(x_1),\cdots,\alpha(x_{n-1}),\{y_1,\cdots,y_n\}_P\}_P&=\displaystyle\sum_{i=1}^{n-1}\{P(\alpha(x_1)),\cdots,\alpha(x_i),\cdots,P(\alpha(x_{n-1})),P(\{y_1,\cdots,y_n\}_P)\}\\&+\{P(\alpha(x_1)),\cdots,P(\alpha(x_{n-1})),\{y_1,\cdots,y_n\}_P\}\\&=\displaystyle\sum_{i=1}^{n-1}\{P(\alpha(x_1)),\cdots,\alpha(x_i),\cdots,P(\alpha(x_{n-1})),\{P(y_1),\cdots,P(y_n)\}\}\\&+\displaystyle\sum_{i=1}^n\{P(\alpha(x_1)),\cdots,P(\alpha(x_{n-1})),\{P(y_1),\cdots,y_i,\cdots,P(y_n)\}\}\\&=\displaystyle\sum_{i=1}^{n-1}\displaystyle\sum_{j=1}^{n-1}(-1)^{|X|^{n-1}|Y|^{j-1}}\{\alpha(P(y_1)),\cdots,[P(x_1),\cdots,x_i,\cdots,P(x_{n-1}),P(y_j)]^C,\cdots,\alpha(P(y_n))\}\\&+(-1)^{|X|^{n-1}|Y|^{n-1}}\{\alpha(P(y_1)),\cdots,\alpha(P(y_{n-1})),\{P(x_1),\cdots,x_i,\cdots,P(x_i),P(y_n)\}\}\\&+ \displaystyle\sum_{i=1}^n\displaystyle\sum_{j=1}^{i-1}(-1)^{|X|^{n-1}|Y|^{j-1}}\{\alpha(P(y_1)),\cdots,[P(x_1),\cdots,P(x_{n-1}),P(y_j)]^C,\cdots,\alpha(y_i),\cdots,\alpha(P(y_n))\}\\&+ \displaystyle\sum_{i=1}^n\displaystyle\sum_{j=i+1}^{n-1}(-1)^{|X|^{n-1}|Y|^{j-1}}\{\alpha(P(y_1)),\cdots,\alpha(y_i),\cdots,[P(x_1),\cdots,P(x_{n-1}),P(y_j)]^C,\cdots,\alpha(P(y_n))\}\\&+\displaystyle\sum_{i=1}^n(-1)^{|X|^{n-1}|Y|^{i-1}}\{\alpha(P(y_1)),\cdots,[P(x_1),\cdots,P(x_{n-1}),y_j]^C,\cdots,\alpha(P(y_n))\}\\&+(-1)^{|X|^{n-1}|Y|^{n-1}}\{\alpha(P(y_1)),\cdots,\alpha(y_i),\cdots,\alpha(P(y_{n-1})),\{P(x_1),\cdots,P(x_{n-1}),P(y_n)\}\}. 
\end{align*}}
By using \eqref{n-Hom-pre-Lie-sup 1}, a direct computation gives that
\begin{align*}
\{\alpha(x_1),\cdots,\alpha(x_{n-1}),\{y_1,\cdots,y_n\}_P\}_P&=\displaystyle\sum_{i=1}^{n-1}(-1)^{|X|^{n-1}|Y|^{i-1}}\{\alpha(y_1),\cdots,\alpha(y_{i-1}),\{x_1,\cdots,x_{n-1},y_i\}_P,\alpha(y_{i+1}),\cdots,\alpha(y_n)\}_P\\&+(-1)^{|X|^{n-1}|Y|^{n-1}}\{\alpha(y_1),\cdots,\alpha(y_{n-1}),\{x_1,\cdots,x_{n-1},y_n\}_P\}_P,
\end{align*}
 which implies that $\{\cdot,\cdots,\cdot\}_P$ satisfies the condition \eqref{n-Hom-pre-Lie-sup 1}. By the same way, we show that the condition \eqref{n-Hom-pre-Lie-sup 2} satisfies by $\{\cdot,\cdots,\cdot\}_P$. Then $\{\cdot,\cdots,\cdot\}_P$ gives an $n$-Hom-pre-Lie superalgebra structure on $\mathcal{A}$.  
\end{proof}
\begin{pro}
Let $(P_1, P_2)$ be a pair of commuting Rota-Baxter operators (of weight zero) on an $n$-Hom-Lie
superalgebra $(\mathcal{A}, [\cdot, \cdots , \cdot],\alpha)$. Then $P_2$ is a Rota-Baxter operator (of weight zero) on the associated $n$-Hom-pre-Lie superalgebra defined by
$$\{x_1,\cdots,x_n\}=[P_1(x_1),\cdots,P_1(x_{n-1}),x_n],\;\forall x_i\in\mathcal{H}(\mathcal{A}),1\leq i\leq n.$$

\end{pro}

\begin{proof}
For any $x_i\in\mathcal{H}(\mathcal{A}),\;1\leq i\leq n$, we have
\begin{align*}
   &\quad \alpha \circ P_2=P_2\circ \alpha\\
     &\quad \;\{P_2(x_1),\cdots,P_2(x_n)\}=[P_1(P_2(x_1)),\cdots,P_1(P_2(x_{n-1})),P_2(x_n)]\\
     &=[P_2(P_1(x_1)),\cdots,P_2(P_1(x_{n-1})),P_2(x_n)]\\
     &=P_2\Big([P_2(P_1(x_1)),\cdots,P_2(P_1(x_{n-1})),x_n]\\
     &+\sum_{i=1}^{n-1}(-1)^{n-i}(-1)^{|x_i||X|^n_{i+1}}[P_2(P_1(x_1)),\cdots,\widehat{P_2(P_1(x_i))},\cdots,P_2(P_1(x_{n-1})),P_1(x_i)]\Big)\\
     &=P_2\Big(\{P_2(x_1),\cdots,P_2(x_{n-1}),x_n\}\\
      &+\sum_{i=1}^{n-1}(-1)^{i+1}(-1)^{|x_i||X|^n_{i+1}}\{x_i,P_2(x_1),\cdots,\widehat{P_2(x_i)},\cdots,P_2(x_{n-1}),P_2(x_n)\}\Big).
\end{align*}
Then $P_2$ is a Rota-Baxter operator (of weight zero) on the $n$-Hom-pre-Lie superalgebra $(A, \{\cdot, \cdots, \cdot\},\alpha)$.
\end{proof}
\section{$n$-Hom-pre-Lie superalgebras induced by Hom-pre-Lie superalgebras}
In \cite{Hajjaji-Chtioui-Mabrouk-Makhlouf}, the authors introduced the construction of an $(n+1)$-pre-Lie algebra from an $n$-pre-Lie algebra using the trace map. In this section we generalize this construction to the super-Hom case by a new approach which is the construction of an $n$-Hom-pre-Lie superalgebras from a Hom-pre-Lie superalgebras, the same work has been studied in the Lie case ( see \cite{Mabrouk-Ncib}). We start with the data of an even super-skew-symmetric $(n-2)$-linear form $\Phi:\mathcal{A}\times\cdots\times\mathcal{A}\to\mathbb{K}$ (i.e. $\Phi(x_1,\cdots,x_{n-2})=0,\;\forall x_i\in\mathcal{H}(\mathcal{A}),\;1\leq i\leq n-2$, where $|x_1|+\cdots+|x_{n-2}|\equiv 1[2]$) and we define from this form an $n$-linear map which is super-skew-symmetric on the first $(n-1)$ variables. Let us define a Hom-pre-Lie superalgebra as a triple $(\mathcal{A},\circ,\alpha)$ consisting of a $\mathbb{Z}_2$-vector space $\mathcal{A}$, an even bilinear map $\circ:\mathcal{A}\times\mathcal{A}\to\mathcal{A}$ and an even linear map $\alpha:\mathcal{A}\to\mathcal{A}$, such that the following condition hold:
\begin{equation}\label{hom-pre-lie}
\mathfrak{ass}(x,y,z)-(-1)^{|x||y|}\mathfrak{ass}(y,x,z)=0,\;\forall x,y,z\in\mathcal{H}(\mathcal{A}),   
\end{equation}
where, $\mathfrak{ass}(x,y,z)=\alpha(x)\circ(y\circ z)-(x\circ y)\circ\alpha(z)$. \\

Let $(\mathcal{A},\circ,\alpha)$ be a Hom-pre-Lie superalgebra. Define the $n$-ary product as follows: 
\begin{equation}\label{phi-trace}
 \{x_1,\cdots,x_{n-1},x_n\}_\Phi=\displaystyle\sum_{k=1}^{n-1} (-1)^{k+1}(-1)^{|x_k||X|^{n-1}_{k+1}}\Phi(x_1,\cdots,\hat{x_k},\cdots,x_{n-1})(x_k\circ x_n), 
\end{equation}
for all $x_k\in\mathcal{H}(\mathcal{A}),\;1\leq k\leq n$.\\

It is clear that $\{\cdot,\cdots,\cdot\}_\Phi$ is an even $n$-linear map.
\begin{pro}
The $n$-ary product $ \{\cdot,\cdots,\cdot\}_\Phi$ is super-skew-symmetric on the first $(n-1)$ variables.
\end{pro}
\begin{proof}
Let $x_1,\cdots,x_n\in\mathcal{H}(\mathcal{A})$, then for all $ i\in\{1,\cdots,n-2\}$, we have
{\small\begin{align*}
 \{x_1,\cdots,x_i,x_{i+1},\cdots,x_{n-1},x_n\}_\Phi&=\displaystyle\sum_{1\leq k\neq i,i+1<n-1}  (-1)^{k+1}(-1)^{|x_k||X|^{n-1}_{k+1}}\Phi(x_1,\cdots,x_i,x_{i+1},\cdots,\hat{x_k},\cdots,x_{n-1})(x_k\circ x_n)\\&+(-1)^{i+1}(-1)^{|x_i||X|^{n-1}_{i+1}}\Phi(x_1,\cdots,\hat{x_i},x_{i+1},\cdots,x_{n-1})(x_i\circ x_n)\\&+(-1)^i(-1)^{|x_{i+1}||X|^{n-1}_{i+2}}\Phi(x_1,\cdots,x_i,\widehat{x_{i+1}},\cdots,x_{n-1})(x_i\circ x_n)\\&=-(-1)^{|x_i||x_{i+1}|}\Big(\displaystyle\sum_{1\leq k\neq i,i+1<n-1}  (-1)^{k+1}(-1)^{|x_k||X|^{n-1}_{k+1}}\Phi(x_1,\cdots,x_{i+1},x_i,\cdots,\hat{x_k},\cdots,x_{n-1})(x_k\circ x_n)\\&+(-1)^i(-1)^{|x_i||X|^{n-1}_{i+2}}\Phi(x_1,\cdots,\hat{x_i},x_{i+1},\cdots,x_{n-1})(_{i+1}\circ x_n)\\&+(-1)^{i+1}(-1)^{|x_i||X|^{n-1}_{i+1}}\Phi(x_1,\cdots,x_i,\widehat{x_{i+1}},\cdots,x_{n-1})_{i+1}\circ x_n)\Big)\\&= -(-1)^{|x_i||x_{i+1}|} \{x_1,\cdots,x_{i+1},x_i,\cdots,x_{n-1},x_n\}_\Phi,
\end{align*}}
which gives that $\{\cdot,\cdots,\cdot\}_\Phi$ is super-skew-symmetric on the first $(n-1)$ terms.
\end{proof}
\begin{thm}\label{n_Hom-pre-super-ind-Hom-pre-super}
Let $(\mathcal{A},\circ,\alpha)$ be a Hom-pre-Lie superalgebra and $\Phi:\mathcal{A}\times\cdots\times\mathcal{A}\to\mathbb{K}$ be an even $(n-2)$-linear super-skew-symmetric form. Then  $(\mathcal{A},\{\cdot,\cdots,\cdot\}_\Phi,\alpha)$ is an $n$-Hom-pre-Lie superalgebra if and only if: 
\begin{align}
 \Phi(x_1\cdots,x_i,y\circ z,x_{i+1},\cdots,x_{n-3})&=0,\;\forall x_i,y,z\in\mathcal{H}(\mathcal{A}),\;1\leq i\leq n-3,\label{cond-induced-1}\\
 \Phi\circ(\alpha\otimes Id \otimes\cdots\otimes Id)&=\Phi,\label{cond-induced-2}\\
 \Phi\wedge\delta\Phi_X&=0,\;\forall X=(x_1,\cdots,x_{n-3})\in\wedge^{n-3}\mathcal{H}(\mathcal{A}),\label{cond-induced-3}
\end{align}
where $$\Phi\wedge\delta\Phi_X(Y)=\displaystyle\sum_{k=1}^{n-1}(-1)^{k+1}(-1)^{|y_k||Y|^{n-1}_{k+1}}\Phi(y_1,\cdots,\hat{y_k},\cdots,y_{n-1})\Phi(X,y_k),\;\forall Y=(y_1,\cdots,y_{n-1})\in\wedge^{n-1}\mathcal{H}(\mathcal{A}),$$
and $\{\cdot,\cdots,\cdot\}_\Phi$ is defined by Eq. \eqref{phi-trace}. We shall say that $(\mathcal{A},\{\cdot,\cdots,\cdot\}_\Phi,\alpha)$ is induced by $(\mathcal{A},\circ,\alpha)$.
\end{thm}
\begin{lem}
If an even super-skew-symmetric $(n-2)$-linear form $\Phi:\mathcal{A}\times\cdots\times\mathcal{A}\to\mathbb{K}$ satisfies condition \eqref{cond-induced-1}, then it satisfies the following condition
\begin{equation}\label{cond-induced-1-equiv}
\Phi(x_1,\cdots,x_i,\{y_1,\cdots,y_n\}_\Phi,x_{i+1},\cdots,x_{n-3})=0,    
\end{equation}
for all $x_i,y_j\in\mathcal{H}(\mathcal{A}),\;(i,j)\in\{1,\cdots,n-3\}\times\{1,\cdots,n\}.$
\end{lem}
\begin{proof}
This is a direct computation, by using the expression of $\{\cdot,\cdots,\cdot\}_\Phi$.
\end{proof}

\begin{proof}[Proof of Theorem \ref{n_Hom-pre-super-ind-Hom-pre-super}]

Let $x_i,y_i\in\mathcal{H}(\mathcal{A}),\;1\leq i\leq n$.\\

On the one hand, we have:
\begin{align*}
M&=\{\alpha(x_1),\cdots,\alpha(x_{n-1}),\{y_1,\cdots,y_n\}_\Phi\}_\Phi\\&=\displaystyle\sum_{i=1}^{n-1}(-1)^{i+1}(-1)^{|x_i||X|^{n-1}_{i+1}}\Phi(\alpha(x_1),\cdots,\widehat{\alpha(x_i)},\cdots,\alpha(x_{n-1}))(\alpha(x_i)\circ\{y_1,\cdots,y_n\}_\Phi)\\&=\displaystyle\sum_{i=1}^{n-1}\displaystyle\sum_{j=1}^{n-1}(-1)^{i+j}(-1)^{|x_i||X|^{n-1}_{i+1}+|y_j||Y|_{j+1}^{n-1}}\Phi(\alpha(x_1),\cdots,\widehat{\alpha(x_i)},\cdots,\alpha(x_{n-1}))\Phi(y_1,\cdots,\hat{y_j},\cdots,y_{n-1})(\alpha(x_i)\circ(y_j\circ y_n)).
\end{align*}
On the other hand, we have:
\begin{align*}
N&=\displaystyle\sum_{j=1}^{n-1}(-1)^{|X|^{n-1}|Y|^{j-1}}\{\alpha(y_1),\cdots,\alpha(y_{j-1}),[x_1,\cdots,x_{n-1},y_j]^C_\Phi,\alpha(y_{j+1}),\cdots,\alpha(y_n)\}_\Phi\\&+(-1)^{|X|^{n-1}|Y|^{n-1}}\{\alpha(y_1),\cdots,\alpha(y_{n-1}),\{x_1,\cdots,x_{n-1},y_n\}_\Phi\}_\Phi\\&=\displaystyle\sum_{j=1}^{n-1}\displaystyle\sum_{i=1}^{n-1}(-1)^{n-i}(-1)^{|X|^{n-1}|Y|^{j-1}}(-1)^{|x_i|(|X|^{n-1}_{i+1}+|y_j|)}\{\alpha(y_1),\cdots,\alpha(y_{j-1}),\{x_1,\cdots,\hat{x_i},\cdots,x_{n-1},y_j,x_i\}_\Phi,\alpha(y_{j+1}),\cdots,\alpha(y_n)\}_\Phi\\&+\displaystyle\sum_{j=1}^{n-1}(-1)^{|X|^{n-1}|Y|^{j-1}}\{\alpha(y_1),\cdots,\alpha(y_{j-1}),\{x_1,\cdots,x_{n-1},y_j\}_\Phi,\alpha(y_{j+1}),\cdots,\alpha(y_n)\}_\Phi\\&+(-1)^{|X|^{n-1}|Y|^{n-1}}\{\alpha(y_1),\cdots,\alpha(y_{n-1}),\{x_1,\cdots,x_{n-1},y_n\}_\Phi\}_\Phi\\&=N_1+N_2+N_3,    
\end{align*}
where
{\small\begin{align*}
 N_1&=\displaystyle\sum_{j=1}^{n-1}\displaystyle\sum_{i=1}^{n-1}(-1)^{n-i}(-1)^{|X|^{n-1}|Y|^{j-1}}(-1)^{|x_i|(|X|^{n-1}_{i+1}+|y_j|)}\{\alpha(y_1),\cdots,\alpha(y_{j-1}),\{x_1,\cdots,\hat{x_i},\cdots,x_{n-1},y_k,x_i\}_\Phi,\alpha(y_{j+1}),\cdots,\alpha(y_n)\}_\Phi\\&=\displaystyle\sum_{j=1}^{n-1}\displaystyle\sum_{i=1}^{n-1}\displaystyle\sum_{k=1}^{j-1}(-1)^{n-i+k+1}(-1)^{\gamma_{ij}}\Phi(\alpha(y_1),\cdots,\widehat{\alpha(y_k)},\cdots,\alpha(y_{j-1}),\{x_1,\cdots,\hat{x_i},\cdots,x_{n-1},y_j,x_i\}_\Phi,\cdots,\alpha(y_{n-1}))(\alpha(y_k)\circ\alpha(y_n))\\&+\displaystyle\sum_{j=1}^{n-1}\displaystyle\sum_{i=1}^{n-1}\displaystyle\sum_{k=j+1}^{n-1}(-1)^{n-i+k+1}(-1)^{\gamma_{ij}+|y_k||X|^{n-1}}\Phi(\alpha(y_1),\cdots,\alpha(y_{j-1}),\{x_1,\cdots,\hat{x_i},\cdots,x_{n-1},y_j,x_i\}_\Phi,\cdots,\widehat{\alpha(y_k)},\cdots,\alpha(y_{n-1}))(\alpha(y_k)\circ\alpha(y_n))\\&+ \displaystyle\sum_{j=1}^{n-1}\displaystyle\sum_{i=1}^{n-1}(-1)^{n-i+j+1}(-1)^{\theta_{ij}}\Phi(\alpha(y_1),\cdots,\widehat{\alpha(y_j)},\cdots,\alpha(y_{n-1}))(\{x_1,\cdots,\hat{x_i},\cdots,x_{n-1},y_j,x_i\}_\Phi\circ\alpha(y_n))
\end{align*}}
where $$\gamma_{ij}=|X|^{n-1}|Y|^{j-1}+|x_i|(|X|^{n-1}_{i+1}+|y_j|)+|y_k|(|X|^{n-1}+|Y|^{n-1}_{k+1}),$$ and $$\theta_{ij}=|X|^{n-1}|Y|^{j-1}+|x_i|(|X|^{n-1}_{i+1}+|y_j|)+|Y|^{n-1}_{j+1}(|X|^{n-1}+|y_j|).$$\\

Using Eq. \eqref{cond-induced-1-equiv}, we notice that the first two terms of the second equality of $N_1$ are zero, which gives that
{\small\begin{align*}
 N_1&=\displaystyle\sum_{j=1}^{n-1}\displaystyle\sum_{i=1}^{n-1}(-1)^{n-i+j+1}(-1)^{\theta_{ij}}\Phi(\alpha(y_1),\cdots,\widehat{\alpha(y_j)},\cdots,\alpha(y_{n-1}))(\{x_1,\cdots,\hat{x_i},\cdots,x_{n-1},y_j,x_i\}_\Phi\circ\alpha(y_n))\\&=\displaystyle\sum_{j=1}^{n-1}\displaystyle\sum_{i=1}^{n-1}\displaystyle\sum_{\substack{k=1 \\ k\neq i}}^{n-1}(-1)^{n-i+j+1}(-1)^{\theta_{ijk}}\Phi(\alpha(y_1),\cdots,\widehat{\alpha(y_j)},\cdots,\alpha(y_{n-1}))\Phi(x_1,\cdots,\hat{x_k},\cdots,\hat{x_i},\cdots,x_{n-1},y_j)((x_k\circ x_i)\circ\alpha(y_n))\\&+   \displaystyle\sum_{j=1}^{n-1}\displaystyle\sum_{i=1}^{n-1}(-1)^{i+j}(-1)^{\theta_{ij}}\Phi((\alpha(y_1),\cdots,\widehat{\alpha(y_j)},\cdots,\alpha(y_{n-1}))\Phi(x_1,\cdots,\hat{x_i},\cdots,x_{n-1})((y_j\circ x_i)\circ\alpha(y_n))\\&=N'_1+N'_2,
\end{align*}}
where $\theta_{ijk}=\theta_{ij}+|x_k|(|X|^{n-1}_{k+1}+|x_i|+|y_j|)$.\\

By the same way, we have:
\begin{align*}
N_2&=\displaystyle\sum_{j=1}^{n-1}(-1)^{|X|^{n-1}|Y|^{j-1}}\{\alpha(y_1),\cdots,\alpha(y_{j-1}),\{x_1,\cdots,x_{n-1},y_j\}_\Phi,\alpha(y_{j+1}),\cdots,\alpha(y_n)\}_\Phi\\&=\displaystyle\sum_{j=1}^{n-1}(-1)^{j+1}(-1)^{|X|^{n-1}|Y|^{j-1}}(-1)^{|Y|^{n-1}_{j+1}(|X|^{n-1}+|y_j|)}\Phi(\alpha(y_1),\cdots,\widehat{\alpha(y_j)},\cdots,\alpha(y_{n-1}))(\{x_1,\cdots,x_{n-1},y_j\}_\Phi\circ \alpha(y_n))\\&=    \displaystyle\sum_{j=1}^{n-1}\displaystyle\sum_{i=1}^{n-1}(-1)^{i+j}(-1)^{\lambda_{ij}}\Phi(\alpha(y_1),\cdots,\widehat{\alpha(y_j)},\cdots,\alpha(y_{n-1}))\Phi(x_1,\cdots,\hat{x_i},\cdots,x_{n-1})((x_i\circ y_j)\circ\alpha(y_n)),
\end{align*}
where $\lambda_{ij}=|X|^{n-1}|Y|^{j-1}+|Y|^{n-1}_{j+1}(|X|^{n-1}+|y_j|)+|x_i||X|^{n-1}_{i+1}=\theta_{ij}+|x_i||y_j|$,\\

and
\begin{align*}
 N_3&= (-1)^{|X|^{n-1}|Y|^{n-1}}\{\alpha(y_1),\cdots,\alpha(y_{n-1}),\{x_1,\cdots,x_{n-1},y_n\}_\Phi\}_\Phi\\&= \displaystyle\sum_{j=1}^{n-1}(-1)^{j+1}(-1)^{|X|^{n-1}|Y|^{n-1}}(-1)^{|y_j||Y|^{n-1}_{j+1}}\Phi(\alpha(y_1),\cdots,\widehat{\alpha(y_j)},\cdots,\alpha(y_{n-1}))(\alpha(y_j)\circ\{x_1,\cdots,x_{n-1},y_n\}_\Phi)\\&=\displaystyle\sum_{j=1}^{n-1}\displaystyle\sum_{i=1}^{n-1}(-1)^{i+j}(-1)^{\nu_{ij}}\Phi(\alpha(y_1),\cdots,\widehat{\alpha(y_j)},\cdots,\alpha(y_{n-1}))\Phi(x_1,\cdots,\hat{x_i},\cdots,x_{n-1})(\alpha(y_j)\circ(x_i\circ y_n)),
\end{align*}
where $\nu_{ij}=|X|^{n-1}|Y|^{n-1}+|y_j||Y|^{n-1}_{j+1}+|x_i||X|^{n-1}_{i+1}=|x_i||y_j|+|y_j||Y|^{n-1}_{j+1}+|x_i||X|^{n-1}_{i+1}$,  since $\Phi$ is even.\\

If we fixed $i,k$ in the expression of $N'_1$, then , we get:
{\small\begin{align*}
N'_1&=\displaystyle\sum_{j=1}^{n-1}\displaystyle\sum_{i=1}^{n-1}\displaystyle\sum_{\substack{k=1 \\ k\neq i}}^{n-1}(-1)^{n-i+j+1}(-1)^{\theta_{ijk}}\Phi(\alpha(y_1),\cdots,\hat{\alpha(y_j)},\cdots,\alpha(y_{n-1}))\Phi(x_1,\cdots,\hat{x_k},\cdots,\hat{x_i},\cdots,x_{n-1},y_j)((x_k\circ x_i)\circ\alpha(y_n))\\&=\displaystyle\sum_{i=1}^{n-1}\displaystyle\sum_{\substack{k=1 \\ k\neq i}}^{n-1}(-1)^{n-i}\Big(\displaystyle\sum_{j=1}^{n-1}(-1)^{j+1}(-1)^{\theta_{ijk}}\Phi(\alpha(y_1),\cdots,\hat{\alpha(y_j)},\cdots,\alpha(y_{n-1}))\Phi(x_1,\cdots,\hat{x_k},\cdots,\hat{x_i},\cdots,x_{n-1},y_j)\Big)((x_k\circ x_i)\circ\alpha(y_n))\\&=0,
\end{align*}}
this from Eq. \eqref{cond-induced-3} and the fact that $\Phi$ is even. Moreover, if we apply condition \eqref{cond-induced-2}, we find
{\small\begin{align*}
M-N&=M-N'_2-N_2-N_3\\&=\displaystyle\sum_{i=1}^{n-1}\displaystyle\sum_{j=1}^{n-1}(-1)^{i+j}(-1)^{|x_i||X|^{n-1}_{i+1}+|y_j||Y|_{j+1}^{n-1}}\Phi(x_1,\cdots,\widehat{\alpha(x_i)},\cdots,x_{n-1})\Phi(y_1,\cdots,\hat{y_j},\cdots,y_{n-1})\\&\Big(\alpha(x_i)\circ(y_j\circ y_n)-(x_i\circ y_j)\circ\alpha(y_n))-(-1)^{|x_i||y_j|}(\alpha(y_j)\circ(x_i\circ y_n)-(y_j\circ x_i)\circ\alpha(y_n)))\Big)\\&=\displaystyle\sum_{i=1}^{n-1}\displaystyle\sum_{j=1}^{n-1}(-1)^{i+j}(-1)^{|x_i||X|^{n-1}_{i+1}+|y_j||Y|_{j+1}^{n-1}}\Phi(x_1,\cdots,\widehat{\alpha(x_i)},\cdots,x_{n-1})\Phi(y_1,\cdots,\hat{y_j},\cdots,y_{n-1})\Big(\mathfrak{ass}(x_i,y_j,y_n)-(-1)^{|x_i||y_j|}\mathfrak{ass}(y_j,x_i,y_n)\Big)\\&=0,
\end{align*}}
The last equality follows from the fact that $(\mathcal{A},\circ,\alpha)$ is a Hom-pre-Lie superalgebra. Then $\{\cdot,\cdots,\cdot\}_\Phi$ satisfies condition \eqref{n-Hom-pre-Lie-sup 1} on $\mathcal{A}$. Similarly, we show that $\{\cdot,\cdots,\cdot\}_\Phi$ satisfies condition \eqref{n-Hom-pre-Lie-sup 2} on $\mathcal{A}$. The theorem is proved.
\end{proof}
Let us given a Hom-pre-Lie superalgebra $(\mathcal{A},\circ,\alpha)$ and an even bilinear form $\Phi:\mathcal{A}\times\mathcal{A}\to\mathbb{K}$ satisfying conditions \eqref{cond-induced-1}, \eqref{cond-induced-2} and \eqref{cond-induced-3}. Then by Theorem \ref{n_Hom-pre-super-ind-Hom-pre-super}, the triple $(\mathcal{A},\{\cdot,\cdot,\cdot,\cdot\}_\Phi,\alpha)$ is a $4$-Hom-pre-Lie superalgebra.

\begin{ex}
Let $\mathcal{A}=\mathcal{A}_{\overline{0}}\oplus\mathcal{A}_{\overline{1}}$ be a two dimensional $\mathbb{Z}_2$-vector space with a basis $\{e_1,e_2\}$, where $\mathcal{A}_{\overline{0}}=<e_1>$ and $\mathcal{A}_{\overline{1}}=<e_2>$. Define on the basis of $\mathcal{A}$ the even bilinear map $\circ:\mathcal{A}\times\mathcal{A}\to\mathcal{A}$ by: 
$$e_2\circ e_2=e_1,$$
and the linear map $\alpha:\mathcal{A}\to\mathcal{A}$ by:
$$\alpha(e_1)=0,\;\;\;\alpha(e_2)=e_2.$$
Then $(\mathcal{A},\circ,\alpha)$ is a Hom-pre-Lie superalgebra.\\

Now, we define the super-skew-symmetric bilinear form $\Phi:\mathcal{A}\times\mathcal{A}\to\mathbb{K}$ by
$$\Phi(e_2,e_2)=\lambda,\;\lambda\in\mathbb{K}.$$ It is obvious that $\Phi$ satisfies the conditions \eqref{cond-induced-1}-\eqref{cond-induced-3}. Then, by Theorem \ref{n_Hom-pre-super-ind-Hom-pre-super}, the triple $(\mathcal{A},\{\cdot,\cdot,\cdot,\cdot\}_\Phi,\alpha)$ is a $4$-Hom-pre-Lie superalgebra, where $\{\cdot,\cdot,\cdot,\cdot\}_\Phi$ defined on the basis of $\mathcal{A}$ by $$\{e_2,e_2,e_2,e_2\}_\Phi=\lambda e_1.$$
\end{ex}

Quite normal and thanks to the importance of the representation theory, any reader asks the following question: Can we also extend this work to representation, i.e. is it possible to construct an $n$-Hom-pre-Lie superalgebra representation from a Hom-pre-Lie superalgebra representation? The answer is yes, but the question requires us to define the representation of a Hom-pre-Lie superalgebra $(\mathcal{A},\circ,\alpha)$ which is defined as a quadruple $(V,l,r,\alpha_V)$ consisting of a $\mathbb{Z}_2$-vector space $V$, two even linear maps $l,r:\mathcal{A}\to gl(V)$ and an even linear map $\alpha_V:V\to V$ such that the following conditions hold:
\begin{align}
 \alpha_V l(x)&=l(\alpha(x))\alpha_V,\;\forall x\in\mathcal{H}(\mathcal{A}),\label{cond-rep-Hom-pre-Lie1}\\
 \alpha_V r(x)&=r(\alpha(x))\alpha_V,\;\forall x\in\mathcal{H}(\mathcal{A}),\label{cond-rep-Hom-pre-Lie2}\\
 l([x,y]^C)\alpha_V&=l(\alpha(x))l(y)-(-1)^{|x||y|}l(\alpha(y))l(x),\;\forall x,y\in\mathcal{H}(\mathcal{A}),\label{cond-rep-Hom-pre-Lie3}\\
 r(\alpha(y))r(x)-r(x\circ y)\alpha_V&=r(\alpha(y))l(x)-(-1)^{|x||y|}r(\alpha(x))l(y),\;\forall x,y\in\mathcal{H}(\mathcal{A}),\label{cond-rep-Hom-pre-Lie4}
\end{align}
where $[x,y]^C=x\circ y-(-1)^{|x||y|}y\circ x,\;\forall x,y\in\mathcal{H}(\mathcal{A})$, which is defined a Hom-Lie superalgebra on $\mathcal{A}$ called the subadjacent Hom-Lie superalgebra of $(\mathcal{A},\circ,\alpha)$. 
\begin{rmk}
Conditions \eqref{cond-rep-Hom-pre-Lie1} and \eqref{cond-rep-Hom-pre-Lie3} are equivalent to saying that $l$ is a representation of the Hom-Lie superalgebra $(\mathcal{A},[\cdot,\cdot]^C,\alpha)$ with respect to $\alpha_V$.
\end{rmk}

In the sequel, we allow to answer the previous question which is summarized by the following proposition.
\begin{pro}\label{repr-induite}
Let $(V,l,r,\alpha_V)$ be a representation of a Hom-pre-Lie superalgebra $(\mathcal{A},\circ,\alpha)$ and $\Phi$ be an even super-skew-symmetric $(n-2)$-linear form satisfying conditions \eqref{cond-induced-1}-\eqref{cond-induced-3}. Then $(V,l_\Phi,r_\Phi,\alpha_V)$ is a representation of the $n$-Hom-pre-Lie superalgebra $(\mathcal{A},\{\cdot,\cdots,\cdot\}_\Phi,\alpha)$, where $l_\Phi,r_\Phi:\mathcal{A}\times\cdots\times\mathcal{A} \rightarrow End(V)$ are two even $(n-1)$-linear maps defined by
\begin{align}
 l_\Phi(x_1,\cdots,x_{n-1})&=\displaystyle\sum_{k=1}^{n-1}(-1)^{k+1}(-1)^{|x_k||X|^{n-1}_{k+1}}\Phi(x_1,\cdots,\hat{x_k},\cdots,x_{n-1})l(x_k),\label{repr-induce-l}\\  
 r_\Phi(x_1,\cdots,x_{n-2},x_{n-1})&=\Phi(x_1,\cdots,x_{n-2})r(x_{n-1}),\label{repr-induce-r}
\end{align}
for al $x_k\in\mathcal{H}(\mathcal{A}),\;1\leq k\leq n-1$.
\end{pro}
\begin{proof}
Let $(V,l,r,\alpha_V)$ be a representation of  $(\mathcal{A},\circ,\alpha)$. Defining the map $\Phi_{\mathcal{A}\oplus V}:\mathcal{A}\oplus V\times\cdots\times\mathcal{A}\oplus V\to\mathbb{K}$ by 
$$\Phi_{\mathcal{A}\oplus V}(x_1+u_1,\cdots,x_{n-2}+u_{n-2})=\Phi(x_1,\cdots,x_{n-2}),\;x_i\in\mathcal{H}(\mathcal{A}),\;u_i\in\mathcal{H}(V),\;1\leq i\leq n-2.$$ Then $\Phi_{\mathcal{A}\oplus V}$ satisfying the conditions \eqref{cond-induced-1}-\eqref{cond-induced-3} on the semi-direct product $n$-Hom-pre-Lie superalgebra  $\mathcal{A}\ltimes_{l,r}^{\alpha_V} V$. Then by \eqref{phi-trace}, we have an $n$-Hom-pre-Lie superalgebra structure on $\mathcal{A}\oplus V$ given by 
{\small\begin{align*}
 \{x_1+u_1,\cdots,x_n+u_n\}_{\Phi_{\mathcal{A}\oplus V}}&=\displaystyle\sum_{k=1}^{n-1}(-1)^{k+1}(-1)^{|x_k||X|^{n-1}_{k+1}}\Phi_{\mathcal{A}\oplus V}(x_1+u_1,\cdots,\widehat{x_k+u_k},\cdots,x_{n-1}+u_{n-1}) \big((x_k+u_k)\circ_{\mathcal{A}\oplus V}(x_n+u_n)\big)\\&=\displaystyle\sum_{k=1}^{n-1}(-1)^{k+1}(-1)^{|x_k||X|^{n-1}_{k+1}}\Phi(x_1,\cdots,\hat{x_k},\cdots,x_{n-1}) \Big(x_k\circ x_n+l(x_k)(u_n)+(-1)^{|x_k||x_n|}r(x_n)(u_k)\Big)\\&=\{x_1,\cdots,x_n\}_\Phi+l_\Phi(x_1,\cdots,x_{n-1})(u_n)+\displaystyle\sum_{k=1}^{n-1}(-1)^{k+1}(-1)^{|x_k||X|^{n-1}_{k+1}}r_\Phi(x_1,\cdots,\hat{x_k},\cdots,x_n)(u_k).  
\end{align*}}
By applying Proposition \ref{carpre}, we deduce that $(V,l_\Phi,r_\Phi,\alpha_V)$ is a representation of $(\mathcal{A},\{\cdot,\cdots,\cdot\}_\Phi,\alpha)$.
\end{proof}
\section{Conclusion}
The results of this paper is to introduce the notion of $n$-Hom-pre-Lie superalgebra and their representation which is very important in several theories among them and the most important the cohomology and deformations theories, we also provide some related results and structures based on Rota-Baxter operators, $\mathcal{O}$-operators and Nijenhuis operators. Under some conditions, an $n$-Hom-pre-Lie superalgebra gives rise to an Hom-pre-Lie superalgebra. In the
future, our plan is to study cohomology and deformations of $n$-Hom-pre-Lie superalgebras.

\end{document}